\documentclass[a4paper,11pt]{amsart}

\usepackage[utf8]{inputenc}
\usepackage[T1]{fontenc}
\usepackage{microtype}
\usepackage{lmodern}

\usepackage[english]{babel}

\usepackage{amsmath}
\usepackage{amssymb}
\usepackage{mathrsfs}

\usepackage[foot]{amsaddr}
\usepackage{amsthm}
\usepackage{esint}
\usepackage{enumerate}

\usepackage{todonotes}

\usepackage{hyperref}
\usepackage[abbrev,backrefs]{amsrefs}

\usepackage{cleveref}

\makeatletter
\def\namedlabel#1#2{\begingroup
   \def\@currentlabel{#2}%
   \label{#1}\endgroup
}
\makeatother

\usepackage[left=2.5cm,right=2.5cm,top=4.5cm,bottom=3.5cm]{geometry}
\theoremstyle{plain}
\newtheorem{proposition}{Proposition}[section]

  \crefname{IntroductionClaim}{proposition}{propositions}
\newtheorem{lemma}[proposition]{Lemma}

\newtheorem{theorem}[proposition]{Theorem}
\newtheorem*{theorem*}{Theorem}

\theoremstyle{definition}
\newtheorem{definition}[proposition]{Definition}
\newtheorem{openproblem}{Open problem}
\theoremstyle{remark}

\numberwithin{equation}{section}

\newcommand{\defeq}{\triangleq}

\usepackage{constants}
\expandafter\let\expandafter\oldproof\csname\string\proof\endcsname
\let\oldendproof\endproof
\renewenvironment{proof}[1][\proofname]{%
  \oldproof[\resetconstant{#1}.]%
}{\resetconstant \oldendproof}
  
\title{Vortex motion for the lake equations}
\date{\today}



\author{Justin Dekeyser}
\address[J. Dekeyser and J. Van Schaftingen]{Universit\'e catholique de Louvain (UCLouvain)\\ 
Institut de Recherche en Math\'ematique et Physique\\
Chemin du Cyclotron 2 bte L7.01.01\\
1348 Louvain-la-Neuve\\
Belgium}
\author{Jean Van Schaftingen}
\email[Jean Van Schaftingen]{Jean.VanSchaftingen@uclouvain.be}

\thanks{\emph{ORCID:} 
  \href{https://orcid.org/0000-0002-8805-6361}{https://orcid.org/0000-0002-8805-6361} (J. Dekeyser),
  \href{https://orcid.org/0000-0002-5797-9358}{https://orcid.org/0000-0002-5797-9358} (J. Van Schaftingen)}

\newcommand{\reals}{\mathbb{R}}
\newcommand{\plane}{\reals^2}
\newcommand{\integers}{\mathbb{N}}
\DeclareMathOperator{\du}{d}
\newcommand{\dif}{{\,\du}}
\newcommand{\crossproduct}[2]{{#1}\times{#2}}
\newcommand{\norm}[2]{\Vert{#1}\Vert_{#2}}
\newcommand{\normal}{\boldsymbol{\nu}}
\newcommand{\tangent}{\boldsymbol{\tau}}

\newcommand{\domain}{D}
\newcommand{\depth}{b}
\newcommand{\closure}[1]{\Bar{#1}}

\newcommand{\island}{I}
\newcommand{\nbislands}{m}

\newcommand{\gradient}{\nabla}
\newcommand{\flipgradient}{\gradient^\perp}
\newcommand{\divergence}{\nabla \cdot}
\newcommand{\curl}{\nabla \times}
\DeclareMathOperator{\dist}{dist}

\DeclareMathOperator{\diam}{diam}
\newcommand{\abs}[1]{\lvert #1 \rvert}
\newcommand{\bigabs}[1]{\bigl\lvert #1 \bigr\rvert}
\newcommand{\Bigabs}[1]{\Bigl\lvert #1 \Bigr\rvert}
\newcommand{\biggabs}[1]{\biggl\lvert #1 \biggr\rvert}
\newcommand{\strabs}[1]{\left \lvert #1 \right \rvert}
\newcommand{\st}{\;\vert\;}

\newcommand{\positivepart}[1]{\left(#1\right)_+}

\newcommand{\vortex}{\omega}
\newcommand{\velocity}{\mathbf{u}}
\newcommand{\boundary}{\partial}
\newcommand{\distance}{\dist}

\newcommand{\energy}{E}

\newcommand{\weakto}{\rightharpoonup}
\newcommand{\vortexstrength}{\Gamma}

\newcommand{\totalvorticity}{\Omega}

\newcommand{\lorentzspace}{\Lambda_{(\ln)_+,1}}
\newcommand{\lorentznorm}[1]{\norm{#1}{\lorentzspace}}

\usepackage{xifthen}
\newcommand{\greenlaplace}[1][]{%
\ifthenelse{\equal{#1}{}}{G_{\domain}}{G_{#1}}%
}
\newcommand{\scalarproduct}[3]{%
\ifthenelse{\equal{#3}{\plane}}{{#1}\cdot{#2}}{({#1}|{#2})_{#3}}%
}
\newcommand{\typscale}[1][]{%
\ifthenelse{\equal{#1}{}}{{\mathfrak{s}}}{{\mathfrak{s}_{#1}}}%
}
\newcommand{\spacetypscale}[1][]{%
\ifthenelse{\equal{#1}{}}{{\rho}}{{\rho^{#1}}}%
}
\newcommand{\timetypscale}[1][]{%
\ifthenelse{\equal{#1}{}}{(\vortexstrength\typscale)}{(\vortexstrength_{#1}\typscale_{#1})}%
}

\newcommand{\scaling}[2][]{%
\ifthenelse{\equal{#1}{}}{\mathcal{Z}_\epsilon\left(#2\right)}{\mathcal{Z}_{#1}\left(#2\right)}%
}

\begin{document}

\begin{abstract}
The lake equations
\begin{equation*}
  \left\{
    \begin{aligned}
      \nabla \cdot 
      \big( b \, \mathbf{u}\big) &= 0
        & & \text{on}\ \mathbb{R}\times D , \\
      \partial_t\mathbf{u} + (\mathbf{u}\cdot\gradient)\mathbf{u} &= -\gradient h
        & & \text{on}\ \mathbb{R}\times D , \\
      \mathbf{u} \cdot \boldsymbol{\nu} &= 0
        & & \text{on}\ \mathbb{R}\times\partial D .
    \end{aligned} 
  \right.
\end{equation*}
model the vertically averaged horizontal velocity in an inviscid incompressible flow of a fluid in a basin whose variable depth $b : D \to [0, + \infty)$ is small in comparison to the size of its two-dimensional projection $D \subset \mathbb{R}^2$.   
When the depth $b$ is positive everywhere in $D$ and constant on the boundary, we prove that the vorticity and energy of solutions of the lake equations whose initial vorticity concentrates at an interior point behaves asympotically a multiple of a Dirac mass whose motion is governed by the depth function $b$.
\end{abstract}

\maketitle

\section{Introduction}

The \emph{lake equations} model an incompressible inviscid flow of a fluid in a lake 
whose velocity varies on distances whose scale is large compared to the depth (shallow water) and is small compared to the speed of gravity waves (small \emph{Froude number}: \(\mathrm{Fr} \ll 1 \)) \cite{Camassa_Holm_Levermore_1997}*{(1.1)}.
Mathematically, the lake is modelled by its projection of its volume on a horizontal planar open set \(\domain\subseteq\plane\) and by a positive depth function \(\depth : \domain \to (0, + \infty)\); the \emph{velocity field} \(\velocity:\reals\times\domain\to\plane\) and the \emph{surface height} \(h:\reals\times\domain\to\reals\)
are governed by the system of equations
  \begin{equation}
  \label{eqLake}
  \left\{
    \begin{aligned}
      \divergence\big( \depth\,\velocity \big) &= 0
        & & \text{on}\ \reals\times\domain , \\
      \partial_t\velocity + (\velocity\cdot\gradient)\velocity &= -\gradient h
        & & \text{on}\ \reals\times\domain , \\
      \velocity\cdot\normal &= 0
        & & \text{on}\ \reals\times\boundary\domain ,
    \end{aligned} 
  \right.
  \end{equation}
where \(\normal\) denotes the outgoing normal vector at the boundary \(\partial D\) of the domain \(\domain\).
The equations \eqref{eqLake} express respectively the conservation of mass, the conservation of momentum and the impermeability of the boundary \(\partial D\). 
In particular when the depth \(b\) is constant on the domain \(D\), the lake equations \eqref{eqLake} reduce to the two-dimensional Euler equations of inviscid incompressible flows.
The lake equations \eqref{eqLake} can be derived formally from the three-dimensional Euler equations \cite{Camassa_Holm_Levermore_1997} and have been justified mathematically in the periodic case \cite{Oliver_1997b}.
They appear in the mean-field limit for the Gross--Pitaevskii equation, which is the Schr\"odinger flow for the Ginzburg--Landau energy, under forcing and pinning \cite{Duerinckx_Serfaty}.
Weak solutions of the Cauchy problem for the lake equations \eqref{eqLake} exist globally \citelist{\cite{Levermore_Oliver_Titi_1996}\cite{Lacave_Pausader_Nguyen_2014}\cite{Munteanu_2012}\cite{Huang_2013}\cite{Oliver_1997}}; these solutions are unique \citelist{\cite{Bresch_Metivier_2006}\cite{Lacave_Pausader_Nguyen_2014}\cite{Munteanu_2012}} and as smooth as the data permits it (\citelist{\cite{Huang_2013} \cite{Levermore_Oliver_1997}\cite{Oliver_1997}} and \cref{appendix_regularity} at the end of the present work).

The \emph{vorticity} \(\vortex=\curl\velocity\) of a flow governed by the lake equations 
\eqref{eqLake} obeys the \emph{vorticity equation}
  \begin{equation}
  \label{eq_vorticity}
        \partial_t\vortex
      + 
        \scalarproduct{\depth\, \velocity}{\gradient\Big( \frac{\vortex}{\depth} \Big)}{\plane} 
      = 0 
    \qquad \text{in \(D\)}
    .
\end{equation}
For the \emph{planar Euler equation}, corresponding to constant depth \(b\), the vorticity equation \eqref{eq_vorticity} has been known since the works of Helmholtz, Kirchhoff and Routh to have singular vortex-point solutions whose vorticity is a linear combination of Dirac deltas whose position is governed by a dynamical system whose Hamiltonian is the \emph{Kirchhoff--Routh stream function} \citelist{\cite{Helmholtz_1867}*{\S 5}\cite{Kirchhoff_1876}*{Zwanzigste Vörlesung, \S 2--3}\cite{Routh_1880}*{\S 23}}.
These vortex point solutions are merely distributional solutions of the Euler equations;  since the works of Scheffer and Shirelman \citelist{\cite{Scheffer_1993}\cite{Shnirelman_1997}\cite{Shirelman_2000}}, the latter are known to exhibit unphysical behaviours in general.
In a seminal work, Marchioro and Pulvirenti have proved mathematically that the singular vortex-point solutions are in fact the \emph{limits of solutions} of the planar Euler equations whose initial data's vorticity concentrates into Dirac masses \cite{Marchioro_Pulvirenti_1983}.

For the lake equations \eqref{eqLake}, Richardson computed by \emph{formal matched asymptotics} that the position \(q : \reals \to \domain\) of a vortex of vorticity \(\Gamma\) and its typical radius \(\varepsilon\) 
should evolve according to the law \cite{Richardson_2000}*{(5.1)}
  \begin{equation}
  \label{eq_Richardson}
      \dot{q}(t) 
    \simeq
      \frac{\Gamma}{4\pi}\,
      \ln\frac{1}{\epsilon}\,
      \bigl(\flipgradient\ln\depth \bigr)\bigl(q(t)\bigr),
  \end{equation}
where the orthogonal gradient is defined \(\flipgradient \ln \depth = (\partial_2 \ln \depth, -\partial_1 \ln \depth)\).
A similar law was derived from axisymmetric Euler flow and verified \emph{experimentally} for vortex dipoles moving towards a planar sloping beach \citelist{\cite{Centurioni_2002}\cite{Peregrine_1998}} and was tested \emph{numerically} on barred beaches  \cite{Buhler_Jacobson_2001} in order to understand the \emph{rip currents} which represent a hazard to swimmers. 
As a consequence of the law \eqref{eq_Richardson}, vortex points should follow at the leading order the level sets of the bathymetry \(\depth\). 
In comparison with the planar Euler equation, the velocity of a vortex depends on its radius and the dominant term is local: it interacts at the leading order neither with the boundary nor with vortices that remain at a positive distance. 
These formal, experimental and numerical results raise the question whether the evolution law \eqref{eq_Richardson} is mathematically the limiting behaviour of families of solutions to the lake equations \eqref{eqLake}.

In the \emph{stationary case} for the lake equations, where the velocity \(\velocity\) does not depend on the time \(t\)  \eqref{eqLake}, there exist families of stationary solutions concentrated at a point of maximal depth or at a point where the irrotational flow generated by a boundary condition of order \(\ln \frac{1}{\varepsilon}\) balances the diverging motion of \eqref{eq_Richardson} \citelist{\cite{deValeriola_VanSchaftingen_2013}\cite{Dekeyser1}\cite{Dekeyser2}}.
(Corresponding results were already known for the planar Euler equations \citelist{\cite{Berger_Fraenkel_1980}\cite{TurkingtonSteady1}\cite{TurkingtonSteady2}\cite{Burton_1988}\cite{VanSchaftingen_Smets_2010}}.)
This approach also yields a rotating singular vortex pair in a rotation-invariant lake \cite{Dekeyser1}.

When \(D = (0, +\infty) \times \reals \) and \(b (r, z) = r\), the lake equations \eqref{eqLake} is in fact the axisymmetric three-dimensional Euler equation.
A single vortex ring is known formally to evolve according to \eqref{eq_Richardson} since the work of Helmholtz and Kelvin \cite{Helmholtz_1867}*{\S 6 and letter from Thompson}, which is a particular case of Da Rios law of evolution of three-dimensional vortices by a binormal curvature flow rescaled by a factor \(\ln \frac{1}{\varepsilon}\) (Da Rios law \citelist{\cite{DaRios_1906}\cite{Ricca_1996}}, see \cite{Banica_Miot_2013}*{\S 2.1} for a derivation in modern formalism).
Benedetto, Caglioti and Marchioro have proved that axisymmetric flows whose initial vorticity concentrates on a vortex ring satisfy asympotically this law \cite{Benedetto_Gaglioti_Marchioro_2000}. For arbitrary filaments, Jerrard and Seis have proved the asymptotic binormal curvature flow under some hypotheses on the solution of the three-dimensional equation \cite{Jerrard_Seis_2017}.

\medbreak

In order to state our main result describing flows whose initial vorticity is concentrating by their bathymetry, we rely on two conserved integral quantities of the flow: the
\emph{vortex circulation} of the flow at time \(t \in \reals\)
  \begin{equation}
  \label{eqVortexCirculation}
      \vortexstrength(t) 
    \defeq 
      \int_{\domain}\vortex(t),
  \end{equation}
and the \emph{kinetic energy} at time \(t \in \reals\),
  \begin{equation}
  \label{eq_def_energy}
      \energy(t) 
    \defeq   
      \frac{1}{2}\int_{\domain}
        \abs{\velocity(t)}^2
        \,
        \depth 
        \,
        ,
  \end{equation}
which are independent of the time \(t \in \reals\) when \(\velocity\) is a classical solution of the lake equation
and for which we will henceforth drop the time-dependence in the notation.

Our main result characterizes the asymptotic behaviour of solutions when the vorticity of the initial data shrinks to a Dirac mass:

\begin{theorem}
\label{truncatedEvolutionLaw}
Let \(\domain\subseteq\plane\) be a bounded domain of class \(C^2\) and \(b \in C^2(\closure{\domain}, (0, +\infty))\).
Assume that \(b\) is constant on each component of \(\partial D\).
If 
\begin{enumerate}[(a)]
 \item \((\velocity^n, h^n)_{n > 0} \in C^1 (\reals \times \Bar{\domain}, \reals^2) \times C (\reals \times \Bar{\domain})\) is a family of classical solutions to the lake equations \eqref{eqLake},
 \item \(\vortex^n (0) \ge 0\) everywhere in \(\domain\),
 \item \(
  \frac{1}{\vortexstrength^n} \vortex^n (0)
 \weakto 
 \delta_{q_0}\) narrowly as measures for some  \(q_0 \in \domain\),
 \item 
  \label{eq_ass_f}
 \(\vortex^n (0) \le C\,\vortexstrength^n \exp \frac{8 \pi \energy^n}{\vortexstrength^n  \,\totalvorticity^n (0)}\) everywhere in \(D\), for some constant $C>0$ independent on \(n\),
\end{enumerate}
then for every \(s \in \reals\),
\begin{align*}
 \frac{1}{\vortexstrength^n} \omega^n \Bigl(\frac{\vortexstrength^n s}{\energy^n}, \cdot\Bigr)
 &\weakto 
 \delta_{q (s)}
 &
 &\text{ and }& 
 \frac{\abs{\velocity^n}^2}{\energy^n} \Bigl(\frac{\vortexstrength^n s}{\energy^n}, \cdot\Bigr)
 &\weakto 
 \delta_{q (s)}
\end{align*} 
narrowly as measures, 
where the function \(q \in C^1 (\reals, D)\) is the unique solution of the Cauchy problem
\begin{equation}
\label{eq_eebahjahb0meezo8heeGheph}
  \left\{
  \begin{aligned}
    q'(s) &= - (\flipgradient\depth^{-1}) \bigl(q(s)\bigr) && \text{\(s \in \reals\)},\\
    q(0)&=q_0 .
  \end{aligned}
  \right.
\end{equation}
\end{theorem}

The assumptions of \cref{truncatedEvolutionLaw} imply in particular that \(\energy^n /\vortexstrength^n \to +\infty\) as \(n \to \infty\).

Here above in \cref{truncatedEvolutionLaw}, the function $\totalvorticity: \domain \to \reals$ stands for the initial
total vorticity
defined for each \(t \in \reals\) by 
  \begin{equation}
  \label{eq_def_totalvorticity}
      \totalvorticity (t) 
    \defeq 
      \int_{\domain} 
        \vortex (t)
        \,
        b .
  \end{equation}
The narrow convergence means explicitly that we assume that for every test function \(\varphi \in C (\Bar{\domain})\)
\[
  \lim_{n \to \infty} \frac{1}{\vortexstrength^n} \int_{\domain} \vortex^n \varphi 
  = \varphi (q_0);
\]
and that we conclude that for every test function \(\varphi \in C (\Bar{\domain})\) and every \(s \in \reals\)
\[
  \lim_{n \to \infty} \frac{1}{\vortexstrength^n} \int_{\domain} \vortex^n \left(\frac{\vortexstrength^n s}{\energy^n} \right) \varphi 
  = \varphi (q (s)).
\]
The narrow convergence of the energy density is similar to the result obtained for two-dimensional incompressible Euler flow by Dávila, del Pino, Musso and Wei \cite{Davila_delPino_Musso_Wei}.

Examples of solutions satisfying the assumptions of \cref{truncatedEvolutionLaw} are 
given by rescaling an initial boundary data since the lake equations \eqref{eqLake} with smooth initial data admit classical solutions. Given a non-negative  function \(f \in C^\infty_c (\plane)\) such that \(\int_{\plane} f = 1\), a sequence of positive numbers \((\varepsilon^n)_{n \in \integers}\) converging to \(0\), a point \(q_0 \in \domain\) and a sequence \((\vortexstrength^n)_{n \in \integers}\) such that for every \(n \in \integers\) \(B (q_0, \varepsilon^n) \subset D\), we define a vorticity \(\vortex^n : \omega \to \reals\),
for each \(n \in \integers\) and \(x \in \domain\)
\begin{equation}
 \vortex^n (0, x) \defeq \frac{\vortexstrength^n}{(\epsilon^n)^2} f \Bigl(\frac{x - q_0}{\varepsilon^n}\Bigr),
\end{equation}
and it can be computed that as \(n \to \infty\)
\begin{align*}
 \energy^n &= b (q_0)\frac{(\vortexstrength^n)^2}{4 \pi} 
 \ln \frac{1}{\varepsilon^n} + O(1)\, 
 &
 \totalvorticity^n(0) &= b(q_0)\,\vortexstrength^n \,\bigl(1  + O (\varepsilon^n)\bigr),
\end{align*}
and thus, by \cref{truncatedEvolutionLaw}, 
\begin{align*}
    \frac{1}{\vortexstrength^n}
    \vortex^n 
    \left(\tfrac{4 \pi s}{\vortexstrength^n b (q_0) \ln \frac{1}{\varepsilon^n}}, \cdot\right)
  &\weakto 
  \delta_{q (s)}
  &\text{ and }&&
  &\text{ and }&
  \frac{\abs{\velocity^n}^2}{b (q_0)(\vortexstrength^n)^2
 \ln \frac{1}{\varepsilon^n}} \left(\tfrac{4 \pi s}{\vortexstrength^n b (q_0) \ln \frac{1}{\varepsilon^n}}, \cdot\right) 
 &\weakto 
 \frac{\delta_{q (s)}}{4 \pi}
\end{align*}
narrowly as measures on \(\Bar{\domain}\), where the motion of \(q\) is governed by \eqref{eq_eebahjahb0meezo8heeGheph}.
If we set \(q_n (t) = q (\tfrac{\vortexstrength^n b (q_0) \ln \frac{1}{\varepsilon^n}}{4 \pi} t)\) and observe that \(b \circ q_n\) is constant,
then \(q_n\) satisfies the equation,
\[
  \left\{
  \begin{aligned}
    q_n'(t) &= \frac{\vortexstrength^n}{4 \pi} \ln \frac{1}{\varepsilon^n}\flipgradient \ln \depth (q_n (t)) && \text{\(t \in \reals\)},\\
    q(0)&=q_0,
  \end{aligned}
  \right.
\]
that is, \(q_n\) is governed by Richardson's law \eqref{eq_Richardson}.

The assumption that the depth \(b\) is constant on each component of the boundary implies that the solution \(q\) of the Cauchy problem obtained in the conclusion of \cref{truncatedEvolutionLaw} remains inside the domain \(D\) and is thus global; the assumption plays an important role in our method, but we do not see any reason for which it should be necessary for the convergence to hold on a time interval in which there is no collision with the boundary.

The assumption that the domain \(\domain\) is simply-connected in \cref{truncatedEvolutionLaw} yields a slightly simpler statement; it will be removed in the sequel under an additional condition that the circulations are controlled by the vortex circulation (see \cref{EvolutionLaw} below). 
Similarly, our proof of \cref{truncatedEvolutionLaw} also covers weak solutions of the lake equations in the vortex formulation.

When the depth \(b\) is a constant function, \cref{truncatedEvolutionLaw} implies that
the vortex is stationary at the time scale \(\vortexstrength^n (0) / \energy^n (0)\); this does contradict the classical planar vortex motion which occurs at a time scale of \(1/\vortexstrength^n (0)\), which is much larger in the regime \(\energy^n (0)/\vortexstrength^n (0)^2 \to +\infty\).

The description of the motion of vortices in \cref{truncatedEvolutionLaw} can be formally written as 
\begin{equation}
\label{eq_Richardson_Energy_Circulation}
    \dot{q} (t)
  \simeq
    -
    \frac{E}{\Gamma} 
    \Bigl(\flipgradient \frac{1}{b}\Bigr) (q (t)).
\end{equation}
An advantage of the formulation \eqref{eq_Richardson_Energy_Circulation} is that the typical radius $\varepsilon$ of the vortex, which is not necessarily preserved or even well-defined a priori along the flow, is replaced by a conserved quantity.

\medbreak

A first step in the proof of \cref{truncatedEvolutionLaw}, is to prove that the vorticity of the solution \(\mathbf{u}_n (t)\) at any time \(t \in \reals\) concentrates as \(n\to+\infty\).
In contrast to other works for the planar Euler equations \cite{Marchioro_Pulvirenti_1984} or cylindrically symmetric Euler equations in the space~\citelist{\cite{Benedetto_Gaglioti_Marchioro_2000}\cite{Marchioro_Pulvirenti_1984}}
in which the geometry of the vorticity region is constrained through its diameter or area, 
we rely on a typical length scale
  \begin{equation} 
  \label{eq_tee1uophouXuPh1EeGhoo9Im}
      \spacetypscale[n] (t) 
    \defeq 
      \exp
        \Big( 
          -\frac
            {4\pi\energy^n}{
              \vortexstrength^n
              \,
              \totalvorticity^n (t)} 
          \Big) ,
  \end{equation}
which is defined in terms of integral quantities related to the flow: the energy \(\energy^n (t)\) defined in \eqref{eq_def_energy} and the circulation \(\vortexstrength^n (t)\) defined in \eqref{eqVortexCirculation}, which are both conserved, and the \emph{total vorticity} \(\totalvorticity^n (t)\) defined in \eqref{eq_def_totalvorticity}
which satisfies 
\begin{equation}
\label{eq_aebeid4Ohdo7Pi0oht9ooJae}
 (\inf_\domain b) \, \Gamma  \le \totalvorticity (t) 
 \le (\sup_{\domain} b) \Gamma.
\end{equation}
The estimate \eqref{eq_aebeid4Ohdo7Pi0oht9ooJae} implies in particular that \(\rho_n (t) \to 0\) uniformly as \(n \to +\infty\). 

In order to show that the vorticity effectively concentrates on balls of radius of the order \(\rho_n (t)\) defined in \eqref{eq_tee1uophouXuPh1EeGhoo9Im}, we rely on the assumption \eqref{eq_ass_f} of \cref{truncatedEvolutionLaw} and on the fact that although the total vorticity \(\totalvorticity_n (t)\) is not conserved, it satisfies an estimate of the form 
\begin{equation}
\label{eq_pho0vion9eiQu2eeNg6yaire}
 \abs{\totalvorticity^n (t) - \totalvorticity^n (0)}
 \le C (\vortexstrength^n)^2 \abs{t}
 = C \abs{\vortexstrength^n}
 \frac{(\vortexstrength^n)^2}{\energy^n} \frac{\energy^n t}{\abs{\vortexstrength^n}}.
 \end{equation}
The proof of \eqref{eq_pho0vion9eiQu2eeNg6yaire} relies on the constancy of the depth \(b\) on connected components of the boundary \(\partial \domain\) (see \cref{conservationDepth}). 
Since \(\energy^n /(\vortexstrength^n)^2 \to \infty\), the estimate \eqref{eq_pho0vion9eiQu2eeNg6yaire} is stronger then the estimate \eqref{eq_aebeid4Ohdo7Pi0oht9ooJae} at time scale \(\vortexstrength^n/\energy^n \).

Our strategy to obtain the equation of motion of the vortex, is to study the \emph{center of vorticity}
\begin{equation}
\label{eq_uo8sheis2iiqu4Phoh3ailai}
    q_n (t) 
  \defeq  
    \frac
      {1}
      {\vortexstrength^n}
    \int_{\domain}
      x 
      \,
      \vortex_n (t, x)
      \dif x
      .
\end{equation}
A formal derivation argument on \eqref{eq_uo8sheis2iiqu4Phoh3ailai} gives the formula
\begin{equation*}
    \dot{q}_n (t)
  \simeq
    \frac{1}{\vortexstrength^n}
    \int_{\domain} 
      \frac{\nabla^\perp b}{b^2} 
      \,
      \psi_n (t) 
      \,
      \vortex_n (t);
\end{equation*}
a suitable study of the asymptotics of the vorticity \(\vortex_n (t)\) and of the stream function \(\psi_n (t)\) shows that the right-hand side behaves asymptotically as \(-(\nabla^\perp b^{-1})  \energy^n(t)/\vortexstrength^n (t)\).
Unfortunately our derivation formula for \(q_n\) would require the identity to be constant on the boundary \(\partial \domain\); we bypass this technical obstacle by considering a modified version of the center of vorticity which is close to the center of vorticity thanks to concentration estimates and some repulsion properties of the boundary.

\medskip

The sequel of the present work is organized as follows. 
In \cref{section_Lake_Model}, we precise the notion of weak solution of the lake equations in the vorticity formulation that we are using in the present work and we explain how the velocity can be reconstructed from the vorticity and the circulation around the boundary components and why the circulation \(\Gamma\) and energy \(E\) are preserved for weak solutions.
In \cref{section_expansion}, we expand the velocity construction formula in terms of the depth \(b\) and the Green function \(G_D\) for the classical Dirichlet problem on \(D\) at a level of precision required by the proof of our main result.
These asymptotics are used in \cref{section_concentration} to obtain various concentration estimates on the vorticity.
In \cref{section_Asymptotic}, we prove our main asymptotic result, after having obtained an asymptotic representation of derivatives of quantities and of the total vorticity \(\totalvorticity\).
\Cref{truncatedEvolutionLaw}, as a first result on the asymptotic behaviour of vortices for the lake equations, raises several open problems for future research that are presented in \cref{section_Problems}.

In a first appendix, we state some variants of classical results for transport equations \cite{DiPerna_Lions_1989} for velocities preserving the density \(b\).
The second appendix is devoted to a classical derivation of regularity results for the lake equations \eqref{eqLake}; this implies in particular that the classical solutions of the lake equations \eqref{eqLake} appearing in \cref{truncatedEvolutionLaw} exist for any smooth initial data.

\section{The lake model}
\label{section_Lake_Model}

\subsection{Weak vortex formulation of the lake equation}

A lake is represented by its projection on a bounded domain \(\domain\subseteq\plane\) of the horizontal plane
endowed and a depth function \(\depth : \domain \to (0, + \infty)\).
We assume that the domain \(\domain\) can be written as 
  \begin{equation*} 
      \domain 
    = 
      \domain_0 
      \setminus \bigcup_{i=1}^\nbislands\island_i ,
  \end{equation*}
where the set \(\domain_0 \subset \plane\) is simply-connected and its boundary is of class \(C^2\) and the islands \(\island_1, \dotsc, \island_\nbislands \subseteq \domain_0\) are disjoint simply-connected compact sets whose boundary is of class \(C^2\). We assume that \(b \in C^2 (\closure{\domain}, (0, + \infty))\). In particular, the depth \(b\) remains bounded away from \(0\) on the domain \(\domain\).

A weak solution of the vorticity formulation of the lake equations will satisfy weakly the following system
\[
\left\{
\begin{aligned}
  \divergence (b \,\velocity) & = 0 & & \text{in \([0, + \infty) \times \domain\)},\\
      \velocity \cdot \normal & = 0 & & \text{on \([0, +\infty) \times \partial \domain\)},\\
              \partial_t\vortex 
          + 
            \divergence
              \big( 
                \velocity
                \, 
                \vortex 
              \big) 
        &= 
          0 
      & 
      &\text{in \([0, +\infty) \times \domain\)},\\
  \curl \velocity & = \vortex & & \text{in \([0, +\infty) \times \domain\)},\\
    \vortex (0, \cdot) & = \vortex_0 & & \text{on \(\domain\)}.
\end{aligned}
\right. 
\]
More precisely, it will fulfill the following definition (see \citelist{\cite{Lacave_Pausader_Nguyen_2014}*{Definition 1.2}\cite{Huang_2013}*{Definition 2.2}}):

\begin{definition}
\label{def_weak_solution_vortex}
Given an initial pair \((\vortex_0, \velocity_0)$ with
$\vortex_0\in L^\infty (\domain, \reals)$ and $\velocity_0\in L^\infty (\domain, \reals^2)\) that satisfies weakly
\[
\left\{
\begin{aligned}
  \divergence (b \, \velocity_0) & = 0 & & \text{in \(\domain\)},\\
  \velocity_0 \cdot \normal & = 0 & & \text{on \(\partial \domain\)},\\
  \curl \velocity_0 & = \vortex_0 & & \text{in \(\domain\)},
\end{aligned}
\right.
\]
a pair \((\vortex, \velocity) \in L^\infty ([0, +\infty) \times D, \reals) \times L^\infty ([0, +\infty), L^2(\domain, \reals^2))\) is \emph{a weak solution of the lake equations in the vorticity formulation} with initial condition \((\vortex_0, \velocity_0)\) whenever 
\begin{enumerate}[(i)]
 \item \label{item_continuity} for every test function \(\varphi \in C^1_c ([0, +\infty) \times  \Bar{\domain})\), one has 
 \[
  \int_0^{+\infty} \int_{\domain} b \,\velocity \cdot \nabla \varphi = 0,
 \]
 \item \label{item_vorticity} for every test function \(\varphi \in C^1_c ([0, +\infty) \times \bar{\domain})\) such that for every \(t \in [0, +\infty)\),
 \(\varphi \vert_{\{t\} \times \partial\domain_0} = 0\) and for every \(i \in \{1, \dotsc, \nbislands\}\), \(\varphi \vert_{\{t\} \times \partial I_i}\) is constant, one has 
\[
    \int_0^{+\infty} 
      \int_{\domain} \bigl(\velocity \cdot \nabla^\perp \varphi -\vortex \,\varphi\bigr) 
  = 
    \int_0^{+\infty} 
      \int_{\domain} \bigl(\velocity_0 \cdot \nabla^\perp \varphi -\vortex_0 \,\varphi\bigr),
\]
\item \label{def_wk_sol_transport}
for every test function \(\varphi \in C^1_c ([0, +\infty) \times \domain)\), one has 
\[
 \int_{\domain}
      \vortex_0\,
      \varphi (0, \cdot)
      +
  \int_0^{+\infty} 
   \int_{\domain}
      \vortex \,\bigl(\partial_t \varphi
  +
      \velocity \cdot \nabla \varphi\bigr)
    = 0.
\]
\end{enumerate}

A pair \((\vortex, \velocity) \in L^\infty (\reals \times D, \reals) \times L^\infty (\reals, L^2(\domain, \reals^2))\) is a weak solution of the vorticity formulation of the lake equations with initial condition \((\vortex_0, \velocity_0)\) whenever the functions \(t \in [0, +\infty) \mapsto (\vortex (t), \velocity (t))\) and \(t \in [0, +\infty) \mapsto (-\vortex (-t), -\velocity (-t))\) are both weak solution to the vorticity formulation with initial condition \((\vortex_0, \velocity_0)\).
\end{definition}

Here and in the sequel, \(\tangent\) denotes the unit tangent vector to the boundary \(\partial \domain\) chosen so that \(\det (\normal, \tangent) = 1\).
The set \(C^1_c ([0, +\infty) \times \Bar{\domain})\) is the set of maps \(\varphi \in C^1([0, +\infty) \times \Bar{\domain})\) such that there exists \(T > 0\) such that \(\varphi = 0\) in \((T, +\infty) \times \domain\); \(C^1_c ([0, +\infty) \times \domain)\) is the set of maps \(\varphi \in C^1([0, +\infty) \times \Bar{\domain})\) such that there exists \(T > 0\) and a compact set \(K \subset \domain \) such that \(\varphi = 0\) in \(([0, +\infty) \times \domain) \setminus ([0, T] \times K)\).

Compared to \cite{Lacave_Pausader_Nguyen_2014}*{Definition 1.2}, \cref{def_weak_solution_vortex} considers fewer test functions in \eqref{def_wk_sol_transport} --- this will not matter eventually (see \cref{proposition_test_functions} below) --- and
incorporates the conservation of circulation around the components of the boundary.
Indeed, it follows from \eqref{item_vorticity} in \cref{def_weak_solution_vortex} that the circulation \(\Gamma_i (t)\) along \(\partial \island_i\) at time \(t\) can be defined by 
\[
    \Gamma_i (t) 
  \defeq 
    \int_{\domain} 
      \bigl(\velocity (t) \cdot \nabla^\perp \varphi -\vortex (t) \,\varphi\bigr),
\]
for any function \(\varphi \in C^1 (\Bar{\domain})\) such that \(\varphi = 1\) on \(\partial \island_i\) and for each \(j \in \{1, \dotsc, \nbislands\} \setminus \{i\}\), \(\varphi = 0\) on \(\partial \island_j\) (see also \cite{Lacave_Pausader_Nguyen_2014}*{(2.12)}).
In view of \eqref{item_vorticity} in \cref{def_weak_solution_vortex}, we have \(\Gamma_i (t) = \Gamma_i (0)\) for almost every \(t \in \reals\).

The surface height \(h\) does not appear in the weak formulation of \cref{def_weak_solution_vortex}, in accordance with the fact that \(\nabla h\) can be recovered in \eqref{eqLake} from \(\velocity\).

The lake equations have at least one global weak solution in the vorticity formulation \citelist{\cite{Lacave_Pausader_Nguyen_2014}*{Lemma 2.11 \& Theorem 1.3}\cite{Levermore_Oliver_Titi_1996}*{Theorem 1}\cite{Bresch_Metivier_2006}*{Theorem 2.2 ii)}}; this solution is unique \citelist{\cite{Bresch_Metivier_2006}*{Theorem 2.2 iii)}\cite{Lacave_Pausader_Nguyen_2014}*{\S 2.3}\cite{Munteanu_2012}*{Theorem 1.1}}.
If moreover \(\vortex_0 \in C^{k, \alpha} (\Bar{\domain})\), one has \(\vortex\in C^{k, \alpha} ([-T, T] \times \domain)\), \(\velocity \in C^{k, \alpha} ([-T, T] \times \domain, \reals^2)\) and \(\velocity \in L^\infty ([-T, T], C^{k + 1, \alpha} (\Bar{\domain}, \reals^2))\) (\cite{Huang_2013}*{Theorem 4.1} and \cref{proposition_classical} below).

An alternative to \cref{def_weak_solution_vortex} is the notion of weak solution for the velocity formulation~\cite{Lacave_Pausader_Nguyen_2014}*{Proposition~2.13}, based on \eqref{eqLake}.
Under regularity assumptions on the domain \(\domain\), the depth function \(b\) and on the initial data \((\vortex_0, \velocity_0)\), both notions are equivalent~\cite{Lacave_Pausader_Nguyen_2014}*{Proposition~A.1}.

\subsection{Velocity reconstruction}
Given a vorticity \(\vortex \in L^1 (\domain)\) 
and 
circulations \(\Gamma_1, \dotsc, \Gamma_\nbislands\), the \emph{velocity reconstruction problem} consists in finding a vector field \(\velocity : D \to \plane\) such that 
\begin{equation}
\label{eq_velocity_reconstruction_problem}
  \left\{
    \begin{aligned}
      \divergence\big(\depth\, \velocity \big) &= 0 & &\text{in}\  \domain,\\      \scalarproduct{\velocity}{\normal}{\plane} &= 0 & &\text{on}\ \boundary\domain,\\
      \curl \velocity  &= \vortex & &\text{in}\  \domain,\\
      \displaystyle \oint_{\partial I_i}\frac{\velocity \cdot \tangent}{\depth} &= \Gamma_i & &
      \text{for \(i \in \{1, \dotsc, \nbislands\}\)},
    \end{aligned}
    \right.
\end{equation}
weakly. 
That is, for every test function \(\varphi \in C^1 (\Bar{\domain})\),
\[
   \int_{\domain} b \,\velocity \cdot \nabla \varphi = 0,
\]
and for every \(\varphi \in C^1 (\Bar{\domain})\) such that \(\varphi \vert_{\partial \domain_0} = 0\) and \(\varphi \vert_{\partial \island_i} = \lambda_i \in \reals\),
\[
      \int_{\domain} \bigl(\velocity \cdot \nabla^\perp \varphi -\vortex \,\varphi\bigr) 
  = \sum_{i = 1}^\nbislands \Gamma_i \,\lambda_i.
\]
The system \eqref{eq_velocity_reconstruction_problem} corresponds thus for a fixed time \(t\) to the continuity equation \eqref{item_continuity} and the definition of vorticity \eqref{item_vorticity} in \cref{def_weak_solution_vortex}.

In view of the divergence-free condition \(\divergence (\depth\,\velocity) = 0\) in \(\domain\)
and of the boundary condition \(\velocity \cdot \normal = 0\) on \(\partial D\), the solution can be written as 
\(\velocity = (\flipgradient \psi)/b\), where \(\psi : \domain \to \reals\) is a scalar \emph{stream function} of the velocity field \(\velocity\) and the function \(\psi\) satisfies the elliptic problem
\begin{equation}
\label{eq_stream_function_problem}
  \left\{
    \begin{aligned}
       - \divergence\big(\depth^{-1} \nabla \psi \big) &= \vortex & &\text{in}\  \domain,\\
      \psi &= 0 & &\text{on}\ \boundary\domain_0,\\
      \psi &\text{ is constant} & & \text{on \(\partial \island_i\) for each \(i \in \{1, \dotsc, \nbislands\}\)},\\
      \int_{\partial \island_i} b^{-1} \frac{\partial \psi}{\partial \normal} &= \Gamma_i
& & \text{for each \(i \in \{1, \dotsc, \nbislands\}\)}. 
    \end{aligned}
    \right.
\end{equation}
Since the ansatz \(\velocity = (\flipgradient \psi)/b\) only defines the stream function \(\psi/b\) up to an additive constant, the boundary condition on \(\partial D_0\) fixes the choice of a particular stream function.

The problem \eqref{eq_stream_function_problem} can be handled by first solving the corresponding classical Dirichlet problem:
\begin{equation}
\label{eq_stream_function_problem_Dirichlet}
  \left\{
    \begin{aligned}
       - \divergence\big(\depth^{-1} \nabla \psi \big) &= \vortex & &\text{in}\  \domain,\\
      \psi &= 0 & &\text{on}\ \boundary\domain. 
    \end{aligned}
    \right.
\end{equation}
Since the function \(\depth\) is smooth and bounded from above and from below on \(\domain\), one has the classical result:

\begin{proposition}
\label{proposition_stream_function_Dirichlet}
For every \(p \in (1, +\infty)\), there exists a linear continuous
operator \(\mathcal{G}_b : L^p (\domain) \to W^{2, p} (\domain, \reals^2)\) such that for
every \(\vortex \in L^p (\domain)\), the function \(\mathcal{G}_b[\vortex] \) is a weak solution of the problem \eqref{eq_stream_function_problem_Dirichlet}.
\end{proposition}

\begin{proof}
See for example \cite{Gilbarg_Trudinger_1998}*{theorem 9.15}.
\end{proof}

We now describe the solution to \eqref{eq_stream_function_problem} in terms of \eqref{eq_stream_function_problem_Dirichlet} (see also \cite{Levermore_Oliver_Titi_1996}*{\S 3}; in the case of the planar Euler equations where \(b\) is constant, see
\citelist{%
  \cite{Lin_1941}
  \cite{Koebe_1918}*{\S 6}
  \cite{VanSchaftingen_Smets_2010}*{(45)}
}%
).

\begin{proposition}
\label{proposition_island_stream}
For every \(i \in \{1, \dotsc, \nbislands\}\), there exists a unique function \(\psi_i \in C^2 (\Bar{\domain})\) that solves  \eqref{eq_stream_function_problem} with \(\vortex = 0\), \(\Gamma_i = 1\) and \(\Gamma_j = 0\) when \(i \ne j\).
\end{proposition}
\begin{proof}
For every \(i \in \{1, \dotsc, \nbislands\}\), let \(\varphi_i \in C^2 (\Bar{\domain})\) be a classical solution to the Dirichlet problem
\begin{equation}
\label{eq_prolem_phi_i}
  \left\{
    \begin{aligned}
       - \divergence\big(\depth^{-1} \nabla \varphi_i \big) &= 0 & &\text{in}\  \domain,\\
      \varphi_i &= \delta_{ij} & &\text{on \(\partial \island_j\) for each \(j \in \{1, \dotsc, \nbislands\}\)},
      \\
      \varphi_i &= 0 & &\text{on \(\partial D_0\)},
    \end{aligned}
    \right.
\end{equation}
where \(\delta_{ij}\) is the Kronecker delta, that is, \(\delta_{ij} = 1\) whenever \(i = j\) and \(\delta_{ij} = 0\) otherwise. 
Since the functions \(\varphi_1, \dotsc, \varphi_{\nbislands}\) are by construction linearly independent and since the domain \(\domain\) is connected, the matrix \((D_{ij})_{1 \le i, j \le \nbislands}\) defined by
\begin{equation}
  \label{eq_def_Dij}  
    D_{ij} 
  \defeq 
    \int_{\domain} 
      \frac
        {
            \nabla \varphi_i 
          \cdot 
            \nabla \varphi_j
        }
        {b}
\end{equation}
is positive-definite and thus invertible. Let \((D^{-1}_{ij})_{1 \le i, j \le \nbislands}\) denote the inverse of this matrix. 
For every \(i \in \{1, \dotsc, \nbislands\}\), we define the function \(\psi_i : \domain \to \reals\) by
\begin{equation*}
    \psi_i 
  \defeq 
    \sum_{j = 1}^{\nbislands} D_{ij}^{-1} \varphi_j .
\end{equation*}
The function \(\psi_i\) satisfies the equation \(-\divergence (b^{-1} \nabla \psi_i) = 0\), the boundary condition \(\psi_i = 0\) on \(\partial \domain_0\) and \(\psi_i\) is constant on each \(\partial \island_j\).
Finally, we compute, in view of the boundary conditions satisfied by \(\varphi_j\) and the definition of \(\psi_i\):
\begin{equation*}
  \int_{\partial \island_j} 
    \frac{1}{b} 
    \frac{\partial \psi_i}{\partial \normal}
  = 
    \sum_{\ell = 1}^{\nbislands}
      D_{i \ell}^{-1}
      \int_{\partial \domain} 
        \frac{\varphi_j}{b} \frac{\partial \varphi_\ell}{\partial \normal}
  =
    \sum_{\ell = 1}^{\nbislands}
      D_{i\ell}^{-1}
      \int_{\domain} 
        \frac{\nabla \varphi_j \cdot \nabla \varphi_\ell}{b}
  =
    \sum_{\ell = 1}^{\nbislands}
      D_{i\ell}^{-1} D_{\ell j}
  = 
    \delta_{ij}.
  \qedhere
\end{equation*}
\end{proof}

\begin{proposition}
\label{proposition_Green_stream_function}
For every \(p \in (1, +\infty)\), there exists a linear continuous
operator \(\mathcal{K}_b : L^p (\domain) \to  W^{2, p} (\domain)\) such that for
every \(\vortex \in L^p (\domain)\),
the function \(\mathcal{K}_b[\vortex]\) satisfies problem \eqref{eq_stream_function_problem} weakly with \(\Gamma_1 = \dotsb = \Gamma_{\nbislands} = 0\).
\\
Moreover, there exists a smooth function \(Q_b \in C^2 (\closure{\domain} \times \closure{\domain}, \reals^2)\) such that for each \(\vortex \in L^p (\domain)\), 
\begin{equation*}
    \mathcal{K}_b [\vortex]
  = 
      \mathcal{G}_b[\vortex] 
    + 
      \int_{\domain} 
        Q_b (\cdot, y) 
        \,
        \vortex (y) 
        \dif y .
\end{equation*}
\end{proposition}

\begin{proof}%
[Proof of \cref{proposition_Green_stream_function}]
We define the function \(Q_b : \Bar{\domain} \times \Bar{\domain} \to \reals\) for every \(x, y \in \domain\) by 
\begin{equation*}
    Q_b (x, y) 
  \defeq  
    - \sum_{i, j = 1}^{\nbislands} 
        \varphi_i (x) 
        D^{-1}_{ij} 
        \varphi_j (y),
\end{equation*}
where the functions \(\varphi_i\) and the matrix \(D_{ij}^{-1}\) were defined in \eqref{eq_prolem_phi_i} and  \eqref{eq_def_Dij} in the proof of \cref{proposition_island_stream}, and for each \(\vortex \in L^p (\domain)\),
\begin{equation*}
    \mathcal{K}_b [\vortex] 
  \defeq 
      \mathcal{G}_b [\vortex]
    + 
      \int_{\domain} 
        Q_b (\cdot, y) 
        \,
        \vortex (y) 
        \dif y.
\end{equation*}
By linearity, we have that \(-\divergence (b^{-1} \nabla \mathcal{K}_b[\vortex])
=-\divergence (b^{-1} \nabla \mathcal{G}_b[\vortex]) = \vortex\), that \(\mathcal{K}_b[\vortex] = \mathcal{G}_b[\vortex] = 0\) on \(\partial \domain_0\) and that
\(\mathcal{K}_b[\vortex]\) is constant on each component of the boundary.
Finally, we have for each \(i \in \{1, \dotsc, \nbislands\}\),
\begin{equation*}
\begin{split}
    \int_{\partial I_i} 
      \frac{1}{b} 
      \frac{\partial \mathcal{K}_b[\vortex]}{\partial \normal}
  &= 
      \int_{\partial \domain} 
        \frac{1}{b} \frac{\partial \mathcal{G}_b[\vortex]}{\partial \normal} \varphi_{i} 
    - 
      \sum_{j, \ell = 1}^\nbislands
        D^{-1}_{j \ell} 
        \int_{\partial \domain} 
          \frac{1}{b}\,
          \varphi_i  
          \,
          \frac{\partial \varphi_j}{\partial \normal}
        \int_{\domain} 
          \varphi_\ell \,\vortex
\\
  &= 
      \int_{\domain} \vortex\, \varphi_i
    - 
      \sum_{j,\ell = 1}^\nbislands 
        D^{-1}_{j \ell} 
        \int_{\domain} \frac{\nabla \varphi_i \cdot \nabla \varphi_j}{b} \int_{D} \vortex \,\varphi_\ell
 = 
    0.\qedhere
\end{split}
\end{equation*}
\end{proof}

We deduce from \cref{proposition_island_stream,proposition_Green_stream_function}
that for every \(\vortex \in L^p (\domain)\) and every \(\Gamma_1, \dotsc, \Gamma_{\nbislands} \in \reals\), the solution \(\psi : D \to \reals\) to the problem \eqref{eq_stream_function_problem}
is given for each \(x \in \domain\) by 
\begin{equation}
\label{eq_stream_function_decomposition}
\begin{split}
    \psi (x)
  &=
    \mathcal{K}_b[\vortex] (x)
    +
      \sum_{i = 1}^{\nbislands}
        \Gamma_i
        \psi_i (x)
  \\
  &= 
      \mathcal{G}_b[\vortex] (x)
    +
      \int_{\domain} 
        Q_b (x, y) 
        \,
        \vortex (y)
        \dif y
    +
      \sum_{i = 1}^{\nbislands}
        \Gamma_i
        \,
        \psi_i (x)  .
\end{split}
\end{equation}
The associated velocity field \(\velocity: D \to \reals^2\) is then given for each \(x \in \domain\) by the relation (see also \citelist{\cite{Lacave_Pausader_Nguyen_2014}*{Lemma 2.6}\cite{Levermore_Oliver_Titi_1996}*{Lemma 5}})
\begin{equation}
\label{eq_velocity_reconstruction}
\begin{split}
    \velocity (x)
  &= 
    \frac{1}{\depth(x)}\biggl(  \flipgradient \mathcal{K}_b[\vortex] (x)
    +
      \sum_{i = 1}^{\nbislands}
        \Gamma_i
        \flipgradient \psi_i (x)\biggr)\\
  &=
    \frac{1}{\depth(x)}\biggl(  
      \flipgradient \mathcal{G}_b[\vortex] (x)
    +
      \int_{\domain} 
        \flipgradient Q_b (x, y) 
        \,
        \vortex (y)
        \dif y
    +
      \sum_{i = 1}^{\nbislands}
        \Gamma_i
        \flipgradient \psi_i (x)
        \biggr),
\end{split}
\end{equation}
where we adopt the convention that \(\nabla^\perp\) only acts on the first two-dimensional variable of the function \(Q_b\).

We conclude this section by showing how the kinetic energy defined by \eqref{eq_def_energy} at a fixed time can be computed in terms of the vorticity and the circulations.

\begin{proposition}
\label{proposition_Energy_Formula}
If \(\vortex \in L^p (\domain)\) for some \(p > 1\), then \(\velocity \in L^2 (\domain, \reals^2)\) and 
\[
     E 
  =
    \frac{1}{2} \int_{\domain} \abs{\velocity}^2 b
  =
     \frac{1}{2}
      \int_{\domain}
        \vortex 
        \,
        \mathcal{K}_b[\vortex]
    +
      \sum_{i = 1}^{\nbislands}
        \Gamma_i
        \int_{\domain}
           \vortex
           \, 
           \psi_i
    +
      \sum_{i, j = 1}^{\nbislands}
        \frac{\Gamma_i \, \Gamma_j}{2}
        \int_{\domain \times \domain}
          \frac
            {\nabla \psi_i \cdot \nabla \psi_j}
            {b}.
\]
\end{proposition}
\begin{proof}
In view of the representation formula for the velocity field \eqref{eq_velocity_reconstruction}, we have, 
\begin{equation}
  \label{eq_ohHiis8aaYae8yul4}
  \begin{split}
    E 
 &= 
    \frac{1}{2}
    \int_{\domain} 
      \frac
        {\abs{
            \nabla \mathcal{K}_b[\vortex] 
          + 
            \textstyle \sum_{i = 1}^{\nbislands} \Gamma_i \nabla \psi_i}^2
        }
        {b}
  \\
    &=
      \frac{1}{2}
      \int_{\domain}
        \vortex 
        \,
        \mathcal{K}_b[\vortex]
    +
      \sum_{i = 1}^{\nbislands}
        \Gamma_i
        \int_{\domain}
           \vortex\, \psi_i
    +
      \sum_{i, j = 1}^{\nbislands}
        \frac{\Gamma_i\, \Gamma_j}{2}
        \int_{\domain \times \domain}
          \frac
            {\nabla \psi_i \cdot \nabla \psi_j}
            {b},
            \end{split}
\end{equation}
by integration by parts and by definition of \(\mathcal{K}_b [\vortex]\) in \cref{proposition_Green_stream_function}.
\end{proof}

\subsection{Additional regularity of weak solutions}

We apply the previous results to the regularity of stream functions of weak solutions (see \cite{Bresch_Metivier_2006}*{Theorem 1 i)}):

\begin{proposition}
\label{propositionRegularity}
If \((\vortex, \velocity) \in L^\infty (\reals \times \domain, \reals) \times L^\infty(\reals, L^2 (\domain, \reals^2))\) is a weak solution to the vortex formulation of the lake equations, then for every \(p \in [1, + \infty)\),
we have
\[\velocity \in L^\infty([0, + \infty), W^{1, p} (\domain))\]
and
\[\mathcal{K}_b [\vortex] \in W^{1, \infty} ([0, +\infty) \times \domain, \reals) .\]
\end{proposition}
\begin{proof}
\resetconstant
We first observe that by \cref{proposition_Green_stream_function},
for almost every \(t \in [0, + \infty)\), we have 
\begin{equation}
    \norm{\mathcal{K}_b [\vortex (t)]}{W^{2, p} (\domain)}
  \le
    \C \norm{\vortex (t)}{L^p (\domain)}
 \le
    \Cl{cst_aich1Pughu} \norm{\vortex (t)}{L^\infty (\domain)}
  \le
    \Cr{cst_aich1Pughu} \norm{\vortex}{L^\infty ([0, +\infty) \times \domain)}
    .
\end{equation}
In particular, by taking \(p > 2\), we have by the supercritical Sobolev embedding theorem that \(\mathcal{K}_b[\vortex] \in L^\infty([0, + \infty), W^{1, \infty} (\domain))\) and thus 
\[
 \norm{u}{L^\infty([0, + \infty), W^{1, p} (\domain))}
 \le \C \norm{\vortex}{L^\infty ([0, +\infty) \times \domain)}.
\]
Next, we have in view of the definition of \(\mathcal{K}_b\), for every \(\varphi \in C^\infty_c ((0, +\infty) \times \domain)\), 
\[
  \int_0^{+\infty}\int_{\domain} \mathcal{K}_b [\vortex ]\, \partial_t \varphi 
 =\int_0^{+\infty} \int_{\domain} \vortex  \,\partial_t \mathcal{K}_b [\varphi]
.\]
Using the evolution equation for the vorticity $\vortex$
(\cref{def_weak_solution_vortex} \eqref{def_wk_sol_transport}),
we obtain, since \(\varphi = 0\) on \(\{0\} \times \domain\) and thus \( \mathcal{K}_b [\varphi] = 0\) on \(\{0\} \times \domain\),
\[
\begin{split}
  \biggabs{\int_0^{+\infty}\int_{\domain} \mathcal{K}_b [\vortex] \, \partial_t \varphi }
  &=
    \biggabs{\int_0^{+\infty}\int_{\domain} \vortex  \, \velocity (t) \cdot \nabla \mathcal{K}_b [\varphi]}\\
  &\le 
    \C 
    \norm{\vortex (t)}{L^\infty (\domain)}^2
    \iint_{(0,{+\infty}) \times \domain} \abs{\mathcal{K}_b[\varphi]}\\
& \le     \C 
    \norm{\vortex (t)}{L^\infty (\domain)}^2
    \iint_{(0,{+\infty}) \times \domain} \abs{\varphi}
\end{split}
\]
This implies that the weak derivative $\partial_t \mathcal{K}_b$
belongs to \(L^\infty ([0, +\infty) \times \domain, \reals)\). Therefore we deduce \(\mathcal{K}_b[\vortex] \in W^{1, \infty} ([1, +\infty) \times \domain)\).
\end{proof}

The regularity that we have obtained so far implies that in fact the spatial boundary conditions on the test functions in \cref{def_weak_solution_vortex} \eqref{def_wk_sol_transport} can be completely relaxed.

\begin{proposition}
\label{proposition_test_functions}
If a pair \((\vortex, \velocity)\)
with \(\vortex \in L^\infty ([0, +\infty) \times D, \reals)\)
and \(\velocity\in L^\infty ([0, +\infty), L^2(\domain, \reals^2))\) is a weak solution of the vorticity formulation of the lake equations with initial conditions \(\vortex_0\in L^\infty(\domain)\) and \(\velocity_0\in L^\infty(\domain,\reals^2)\), then 
for every test function \(\varphi \in C^1_c ([0, +\infty) \times \Bar{\domain})\), one has
\[
 \int_{\domain}
      \vortex_0\,
      \varphi (0, \cdot)
      +
  \int_0^{+\infty} 
   \int_{\domain}
      \vortex \,\bigl(\partial_t \varphi
  +
      \velocity \cdot \nabla \varphi\bigr)
    = 0.
\]
\end{proposition}
\begin{proof}
This follows from \cref{propositionRegularity}, \cref{def_weak_solution_vortex} \eqref{def_wk_sol_transport} and \cref{lemma_Transport_TestFunctions}.
\end{proof}

\subsection{Transport of the potential vorticity}
The vorticity equation \eqref{eq_vorticity} can be rewritten as 
\begin{equation} 
  \partial_t\vortex + \divergence \bigl( \velocity\,\vortex \bigr) = 0
\end{equation}
and implies that the vortex circulation \(\Gamma (t)\) defined by \eqref{eqVortexCirculation} of classical solutions of the lake equations \eqref{eqLake} is conserved.
The next proposition shows that this is still the case for weak solutions of the vorticity formulation of the lake equations (\cref{def_weak_solution_vortex}).

\begin{proposition}
\label{proposition_continuity_vortex}
If \((\vortex, \velocity) \in L^\infty (\reals \times \domain, \reals) \times L^\infty( \reals, L^2 (\domain, \reals^2))\) is a weak solution to the vortex formulation of the lake equations, then 
\(\vortex \in C (\reals, L^1 (\domain))\) and for every \(t \in \reals\), 
\[
 \Gamma(t) = \int_{\domain} \vortex (t) = \int_{\domain} \vortex_0 = \Gamma(0).
\]
\end{proposition}

\begin{proof}
We follow \cite{Lacave_Pausader_Nguyen_2014}*{\S 2.3}.
We observe that for every \(t \in \reals\), \(\velocity (t) \in W^{1, 1}_{\mathrm{loc}} (\domain)\) and thus  if \(\varphi \in C^1_c ([0, +\infty) \times \domain)\), we have
\begin{multline}
   \int_0^{+\infty} 
      \int_{\domain} 
          \frac{\vortex}{b}
          \, 
          \bigl(\partial_t \varphi + \divergence (\velocity \varphi)\bigr)
        +
    \int_{\domain} 
      \frac{\vortex_0}{b}
      \,
      \varphi (0, \cdot) 
 \\
  = 
    \int_0^{+\infty} 
      \int_{\domain} 
        \vortex \,\bigl(\partial_t \tfrac{\varphi}{b} + \velocity \cdot \nabla \tfrac{\varphi}{b}\bigr)
    +
      \int_{\domain} 
        \vortex_0
        \,
        \frac{\varphi (0, \cdot)}{b}  
   =
    0.
\end{multline}
By \cref{propositionRegularity}, \cref{proposition_DiPerna_Lions} is applicable to \(f_0 = \vortex_0/b\) and gives the conclusion.
\end{proof}

\subsection{Conservation of energy}
We now consider the total kinetic energy defined by \eqref{eq_def_energy}.
For classical solutions, one can show that the energy equation
\begin{equation}
  \partial_t \bigl(b\,\tfrac{\abs{\velocity}^2}{2}\bigr)
  + \divergence \bigl(b\, \velocity\, \tfrac{\abs{\velocity}^2}{2}\bigr)
  = -\divergence \bigl(b\, \velocity\, h\bigr)
\end{equation}
holds, and consequently, since \(b\, \velocity \cdot \normal = 0\) on the boundary, we have conservation of the total kinetic energy for classical solutions.
The total kinetic energy still remains constant for weak solutions of the vortex formulation of the lake equations (\cref{def_weak_solution_vortex}).

\begin{proposition}[Conservation of energy]
\label{proposition_conservation_energy}
If \((\vortex, \velocity) \in L^\infty (\reals \times \domain, \reals) \times L^\infty( \reals, L^2 (\domain, \reals^2))\) is a weak solution to the vortex formulation of the lake equations, then for almost every \(t \in \reals\),
\[
  E (t) = E (0).
\]
\end{proposition}

The proof of \cref{proposition_conservation_energy} relies on the following derivation formula.

\begin{lemma}
\label{lemma_derivative_quadratic}
Given a weak solution of the vortex formulation of the lake equations
\((\vortex, \velocity) \in L^\infty ([0, +\infty) \times \domain, \reals) \times L^\infty( [0, +\infty), L^2 (\domain, \reals^2))\), we have for every \(\theta \in C^\infty_c ([0, +\infty))\),
\begin{multline}
 \frac{1}{2}\int_0^{+\infty} \biggl(\int_{\domain} \vortex (t)\, \mathcal{K}_b[\vortex (t)]\biggr) \theta' (t) \dif t
 \\
 = -\frac{\theta (0)}{2} \int_{\domain} \vortex_0\, \mathcal{K}_b[\vortex_0] 
 - \int_0^{+\infty} \biggl(\int_{\domain} \vortex (t)\, \partial_t \mathcal{K}_b[\vortex](t)\biggr)\, \theta (t) \dif t.
\end{multline}
\end{lemma}
\begin{proof}
For every \(h \in (0, +\infty)\), we have by a change of variable
\begin{multline}
\label{eq_ohf1aCh7ei}
 \frac{1}{2} \int_0^{+\infty} \biggl(\int_{\domain} \vortex (t)\, \mathcal{K}_b[\vortex (t)]\biggr) \frac{\theta (t + h) - \theta (t)}{h} \dif t
 \\
  = \frac{1}{2} \int_0^{+\infty} \biggl(\int_{\domain} \frac{\vortex (t-h) \,\mathcal{K}_b[\vortex (t-h)]- \vortex (t)\, \mathcal{K}_b[\vortex (t)]}{h} \biggr)\,\theta (t)\dif t
 \\=
 - \frac{1}{2 h} 
 \int_0^{h} \biggl(\int_{\domain} \vortex (t)\, \mathcal{K}_b[\vortex (t)]\biggr) \theta (t) \dif t .
\end{multline}
By Lebesgue's dominated convergence theorem, we have 
\begin{multline}
\label{eq_UCoogh7If9}
 \lim_{h \to 0} \frac{1}{2} \int_0^{+\infty} \biggl(\int_{\domain} \vortex (t)\, \mathcal{K}_b[\vortex (t)]\biggr) \frac{\theta (t + h) - \theta (t)}{h} \dif t
 \\
 =  \frac{1}{2} \int_0^{+\infty} \biggl(\int_{\domain} \vortex (t)\, \mathcal{K}_b[\vortex (t)]\biggr) \theta' (t) \dif t.
\end{multline}
Since \(\vortex \in C([0, +\infty), L^1 (\domain))\) by \cref{proposition_test_functions}, we also have 
\begin{equation}
\label{eq_ohphaep2Ph}
\lim_{h \to 0}\frac{1}{h}
\int_0^{h} \biggl(\int_{\domain} \vortex (t)\, \mathcal{K}_b[\vortex (t)]\biggr) \theta (t) \dif t
=
 \frac{1}{2} \int_{\domain} \vortex_0\, \mathcal{K}_b[\vortex_0] \theta (0).
\end{equation}
For every \(t \in [0, +\infty)\), since 
\[
 \int_{\domain} \vortex (t) \, \mathcal{K}_b[\vortex (t-h)]
 = \int_{\domain} \vortex (t - h) \, \mathcal{K}_b[\vortex (t)],
\]
we have 
\begin{multline*}
  \int_{\domain} 
      \vortex (t-h) 
      \, 
      \mathcal{K}_b[\vortex (t-h)]
    - 
      \vortex (t)
      \, 
      \mathcal{K}_b[\vortex (t)]
=
  \int_{\domain} 
    \bigl(\vortex (t-h) + \vortex (t)\bigr) 
    \bigl(\mathcal{K}_b[\vortex (t-h)] - \mathcal{K}_b[\vortex (t)]\bigr),
\end{multline*}
and thus by the weak convergence of difference quotients to the weak derivative and by \cref{proposition_continuity_vortex}, we obtain 
\begin{multline}
\label{eq_paPh7Chu3g}
  \lim_{h \to 0}
    \frac{1}{2} 
    \int_0^{+\infty} 
      \biggl(
        \int_{\domain} 
          \frac{
              \vortex (t-h)\, \mathcal{K}_b[\vortex (t-h)]
            - 
              \vortex (t)\, \mathcal{K}_b[\vortex (t)]
            }
            {h} 
      \biggr)
      \theta (t)
      \dif t\\
 = 
  \int_0^{+\infty} 
    \biggl(
      \int_{\domain} 
        \vortex (t) 
        \partial_t \mathcal{K}_b[\vortex](t)
    \biggr) 
    \theta (t) 
    \dif t
    .
\end{multline}
The conclusion follows from \eqref{eq_ohf1aCh7ei}, \eqref{eq_UCoogh7If9}, \eqref{eq_ohphaep2Ph} and \eqref{eq_paPh7Chu3g}.
\end{proof}

\begin{proof}[Proof of \cref{proposition_conservation_energy}]
We consider a function \(\theta \in C^\infty_c ([0, + \infty))\).
We want to prove that 
\[
  E (0) \,\theta (0) + \int_0^{+\infty} \theta'(t) E (t) \dif t = 0 .
\]
We rely on the energy formula of \cref{proposition_Energy_Formula}. We first have by \cref{lemma_derivative_quadratic},
\[
  \frac{1}{2} \int_0^{+\infty} \int_{\domain} \vortex (t)\, \mathcal{K}_b [\vortex (t)] \,\theta' (t) \dif t
 = \frac{\theta (0)}{2} \int_{\domain} \vortex_0\, \mathcal{K}_b [\vortex_0] + \int_0^{+\infty} \int_{\domain} \vortex (t)\, \partial_t (\mathcal{K}_b [\vortex] \theta) (t) \dif t ,
\]
since by Leibniz's rule \(\partial_t (\mathcal{K}_b [\vortex] \theta)(t) = \partial_t (\mathcal{K}_b [\vortex])(t) \theta (t) + \mathcal{K}_b[\vortex (t)] \theta'(t)\).
By definition of weak solution of the lake equations in the vorticity formulation (\cref{def_weak_solution_vortex} \eqref{def_wk_sol_transport}), we have then 
\begin{multline}
\label{cst_Cai3chel9e}
    \frac{1}{2} \int_0^{+\infty} \int_{\domain} \vortex (t)\, \mathcal{K}_b [\vortex (t)] \,\theta' (t) \dif t
 \\
 = - 
    \frac{1}{2} \int_{\domain} \vortex_0\, \mathcal{K}_b [\vortex_0] 
    -\int_0^{+\infty} \int_\domain \vortex (t)\, \velocity (t) \cdot \nabla \mathcal{K}_b [\vortex (t)]\, \theta (t) \dif t,
\end{multline}
and, for every \(i \in \{1, \dotsc, \nbislands\}\),
\begin{equation}
\label{cst_nooThah1oa}
  \int_0^{+\infty} \int_{\domain} \vortex (t) \, \psi_i \, \theta' (t) \dif t
 = -\theta (0)\int_{\domain} \vortex_0 \,\psi_i 
   - \int_0^{+\infty} \int_{\domain} \vortex (t)\, \velocity (t) \cdot \nabla \psi_i (t) \, \theta (t) \dif t.
\end{equation}
This implies thus, by combining \cref{proposition_Energy_Formula} with the identities \eqref{cst_Cai3chel9e} and \eqref{cst_nooThah1oa} and the velocity reconstruction formula \eqref{eq_velocity_reconstruction}, that 
\[ 
E (0) \,\theta (0) + \int_0^{+\infty} \theta'(t) E (t) \dif t
     = \int_0^{+\infty} \biggl(\int_{\domain} \vortex (t) \,\velocity (t) \cdot \velocity (t)^\perp\biggr) \theta (t) \dif t
     = 0. \qedhere
\]
\end{proof}

\section{Velocity reconstruction expansion}
\label{section_expansion}

In the sequel, we will need to understand the behaviour of the operator \(\mathcal{K}_b\) appearing in the construction in \cref{proposition_Green_stream_function} of the stream function satisfying \eqref{eq_stream_function_problem}. 
In view of \cref{proposition_Green_stream_function}, 
this can be done through the study of the operator \(\mathcal{G}_b\) associated to the solution of the Dirichlet problem \eqref{eq_stream_function_problem_Dirichlet}, whose existence was given in \cref{proposition_stream_function_Dirichlet}.

\subsection{Construction of the Green function}
We represent the Green function of the Dirichlet problem \eqref{eq_stream_function_problem_Dirichlet} as a perturbation of the Green operator of the classical Laplacian on the same domain with Dirichlet boundary conditions.

\begin{proposition}
\label{proposition_b_Green_function}
There exists a function \(S_b \in C^{0, 1} (\domain \times \domain)\) such that for every \(\vortex \in L^p (\domain)\) and every \(x \in \domain\),
\begin{equation*}
    \mathcal{G}_b [\vortex] (x)
  =
    \int_{\domain} 
      \Bigl(
          G_\domain (x, y) 
          \sqrt{b (x)\, b (y)} 
        + 
          S_b (x, y)
      \Bigr)
      \,
      \vortex (y) 
      \dif y.
\end{equation*}
\end{proposition}

Here \(G_\domain : \domain \times \domain \to \reals\) is the Green function of the Laplacian \(-\Delta\) with Dirichlet boundary conditions on the boundary \(\partial D\), that is, if \(f \in L^p (\domain)\) and if 
\begin{equation*}
  u (x) = \int_{\domain} G_\domain (x, y) f(y) \dif y,
\end{equation*}
then 
\begin{equation*}
\left\{
\begin{aligned}
  -\Delta u & = f & & \text{in \(D\)},\\
  u & = 0 & & \text{on \(\partial D\)}.
\end{aligned}
\right.
\end{equation*}

In particular, \cref{proposition_b_Green_function} implies that the weighted Dirichlet problem \eqref{eq_stream_function_problem_Dirichlet} has a Green function \(G_b : D \times D \to \reals\) defined for each \(x, y \in \domain\) with \(x \ne y\) by 
\begin{equation*}
    G_b (x, y) 
  \defeq 
    G_\domain (x, y) \sqrt{b (x) b (y)} + S_b (x, y),
\end{equation*}
and thus the stream function problem \eqref{eq_stream_function_problem} also has a Green function in view of \cref{proposition_Green_stream_function}.

The proof of \cref{proposition_b_Green_function} will rely on the fundamental estimate, which is a classical consequence of the maximum principle for the Laplacian operator \(-\Delta\).

\begin{proposition}
\label{propositionGreenLogBound}
For every \(x, y \in \domain\)
  \begin{equation*} 
      0 
    \le 
      \greenlaplace(x,y) 
    \leq 
      \frac{1}{2\pi}\ln\frac{\diam (D)}{\abs{x - y}}.
  \end{equation*}
\end{proposition}

\begin{proof}[Proof of \cref{proposition_b_Green_function}]
For each \(y \in \domain\), let \(S_b (\cdot, y) \in W^{1, 2}_0 (D)\) 
be the unique weak solution to the Dirichlet problem
\begin{equation*}
  \left\{
  \begin{aligned}
    -\divergence
      \bigl( 
        \depth^{-1}\gradient S_b (\cdot, y) 
      \bigr) 
  & =     
    -\greenlaplace(\cdot, y)
    \,
    \sqrt{\depth(y)}
    \,
    \Bigl(\Delta \frac{1}{\sqrt{b}}\Bigr)
    &
    &
    \text{in \(\domain\)},
    \\
      S_b (\cdot, y) 
    & =
      0
    &
    & \text{on \(\partial \domain\).}
  \end{aligned}
  \right.
\end{equation*}
Since \(b \in C^2 (\Bar{D}, (0, + \infty))\), 
by classical elliptic regularity estimates 
(see for example \cite{Gilbarg_Trudinger_1998}*{theorem~9.15}), for every \(y \in D\)
we have \(S_b (\cdot, y) \in W^{2, p} (D)\) for every \(p \in (1, +\infty)\) and 
\begin{equation}
\label{estimate_S_W2p}
    \norm{S_b (\cdot, y)}{W^{2, p} (D)}
  \le 
    \C 
    \norm{\greenlaplace(\cdot,y)}{L^p (\domain)}.
\end{equation}

By \cref{propositionGreenLogBound}, we have 
\begin{equation*}
    \norm{\greenlaplace(\cdot,y)}{L^p (\domain)} 
  \le 
    \biggl( \int_{B(0,\diam (\domain))} \Bigl(\ln \frac{\diam \domain}{\abs{z}} \Bigr)^p \dif z \biggr)^\frac{1}{p}
  \le 
    \C.
\end{equation*}
It follows in particular by the classical Sobolev embedding theorem and by \eqref{estimate_S_W2p} that 
\begin{equation*}
 \sup_{y \in D} \norm{\nabla S_b (\cdot, y)}{L^\infty (\domain)} < +\infty.
\end{equation*}

Finally, we observe that if \(\vortex_1, \vortex_2 \in L^p (\domain)\), we have 
\begin{multline*}
      \int_{D} 
        \vortex_1 
        \mathcal{G}_b[\vortex_2]
    - 
      \iint_{D \times D} 
        G (x, y) 
        \,
        \vortex_1 (x) 
        \,
        \vortex_2 (y)  
        \,
        \sqrt{b (x)\, b (y)} 
        \dif x \dif y\\
  =
      \int_{D} 
        \vortex_2 
        \,
        \mathcal{G}_b[\vortex_1]
    - 
      \iint_{D \times D}
        G (x, y) 
        \,
        \vortex_2 (x)
        \,
        \vortex_1 (y)  
        \,
        \sqrt{b (x)\, b (y)} 
        \dif x \dif y
\end{multline*}
and therefore 
\begin{equation*}
  \iint_{D \times D} 
    S_b (x, y)
    \,
    \vortex_1 (x)  
    \,
    \vortex_2 (y) 
    \dif x 
    \dif y
  =
    \iint_{D \times D} 
      S_b (x, y)
      \,  
      \vortex_2 (x)  
      \,\vortex_1 (y) 
      \dif x 
      \dif y
       .
\end{equation*}
It follows that for every \(x, y \in \domain\), \(S_b (x, y) = S_b (y, x)\), and thus the function \(S_b\) is Lipschitz-continuous on \(D \times D\).
\end{proof}

As a consequence of \cref{proposition_b_Green_function}, the velocity field \(\velocity\) admits the integral representation

\begin{proposition}
\label{velocityIntegralExpansion}
There exists a Lipschitz-continuous function \(R_b \in C^{0, 1} (\domain \times \domain)\) such that for every 
\(\vortex \in L^p (\domain)\), one has 
  \begin{equation*}
      \mathcal{K}_b [\vortex] (x)
    = 
        \int_{\domain}
          \bigl(
          \greenlaplace(x,y)
          \,
          \sqrt{\depth(x)\,\depth(y)}
          +
          R_b (x, y)\bigr)
          \,
          \vortex(y)
          \dif y. 
  \end{equation*}
\end{proposition}

\begin{proof}
[Proof of \cref{velocityIntegralExpansion}]
This follows from \cref{proposition_Green_stream_function} and \cref{proposition_b_Green_function} with \(S_b = R_b + Q_b\). 
\end{proof}

\subsection{Estimate on the Green function}

We will also need a version of \cref{propositionGreenLogBound} which is sharper close to the boundary.
\begin{proposition}
\label{proposition_Green_D_upper}
There exists a constant \(C\) such that for every \(x, y \in \domain\),
\begin{equation*}
    G_\domain (x, y)
  \le
    \frac
      {1}
      {4 \pi}
    \ln
      \biggl(
          1
        +
          C
          \frac
            {\dist (x, \partial \domain) \dist (y, \partial \domain)}
            {\abs{x - y}^2}
      \biggr).
\end{equation*}
\end{proposition}

\begin{proof}
This can be obtained by observing that for the unit disk \(\mathbb{D}^2 \subset \plane\), one has for each \(x, y \in \mathbb{D}^2\) such that \(x \ne y\),
\[
  G_{\mathbb{D}^2} (x, y)
  = \frac{1}{4 \pi} \ln \left(1 + \frac{(1 - \abs{x}^2)(1 - \abs{y}^2)}{\abs{x - y}^2}\right),
\]
and by applying conformal mapping techniques as in the proof of \cref{proposition_symmetric_gradient_estimate_boundary}.
\end{proof}

\subsection{Gradient estimates of Green function of the Laplacian}
Our goal now is to obtain estimates on the derivative of the Green function \(G_D\).
A first classical estimate is available \cite{Bramble_Payne_1967}.

\begin{proposition}
\label{proposition_Gradient_bound}
There exists a constant \(C\) such that for every \(x, y \in \domain\), one has
\begin{equation*}
    \abs{\nabla G_\domain (x, y)} 
  \le    
    \frac
      {C}
      {\abs{x - y}}.
\end{equation*}
\end{proposition}

We will need a more refined directional information about the Green function of the Laplacian.
We observe that in view of the definition of the regular part \(H_\domain : \domain \times \domain \to \reals\) for \(x, y \in \domain\) such that \(x \ne y\) as 
\begin{equation}
\label{eq_definition_regular_part}
    H_\domain (x, y) 
  \defeq 
    G_\domain (x, y) - \frac{1}{2\pi} \ln \frac{1}{\abs{x - y}},
\end{equation}
we have for every \(x, y \in \domain\) such that \(x \ne y\)
\begin{equation}
\label{symmetric_Gradient_Identity}
      \nabla G_\domain (x, y) 
    + 
      \nabla G_\domain (y, x)
  = 
    \nabla H_\domain (x, y) + \nabla H_\domain (y, x).
\end{equation}
Here above, \(\gradient\greenlaplace\) denotes the gradient of \(\greenlaplace\) with respect to its first variable.

In view of the regularity properties of the regular part of the Green function, we get 

\begin{proposition}
[Interior symmetric gradient estimate]
\label{proposition_symmetric_gradient_estimate}
For every \(\delta > 0\), there exist \(C > 0\) such that if \(\dist (x, \partial \domain) + \dist (y, \partial \domain) + \abs{x - y} \ge \delta\),
\begin{equation*}
    \abs{
        \nabla G_\domain (x, y) 
      + 
        \nabla G_\domain (y, x)
        }
  \le
    C.
\end{equation*}
\end{proposition}
\begin{proof}
This follows from \eqref{symmetric_Gradient_Identity} and the smoothness of the regular \(H_D\) part of the Green function defined in  \eqref{eq_definition_regular_part}.
\end{proof}

We now investigate what the estimate of \cref{proposition_symmetric_gradient_estimate} becomes near the boundary \(\partial \domain\). We start by observing the Green function of the Laplacian on the half-plane \(\mathbb{R}^2_+ = \big\{ x = (x_1, x_2) \in \plane \st  x_2 >0\big\}\), 
which is given for each \(x = (x_1, x_2), y = (y_1, y_2) \in \domain\) by
  \begin{equation*} 
      \greenlaplace[\mathbb{R}^2_+](x,y) 
    =   
      \frac{1}{4\pi}
        \ln \left(1 + \frac{4 x_2 y_2}{\abs{x-y}^2}\right) .
  \end{equation*}
The gradient of this function with respect to its first variable, is then given by
  \begin{equation*} 
      \gradient \greenlaplace[\mathbb{R}^2_+] (x,y)
    = 
      \frac
        {1}
        {\pi (\abs{x - y}^2 + 4 x_2 y_2)} 
      \Bigl((0, y_2) - 2 x_2 y_2 \frac{x - y}{\abs{x - y}^2}\Bigr)
    .
  \end{equation*}
One computes then that 
  \begin{equation}
  \label{eq_Symmetricpart_halfspace}
        \gradient\greenlaplace[\mathbb{R}^2_+](x,y)
      +
        \gradient\greenlaplace[\mathbb{R}^2_+](y,x)
    = 
      \frac
        {(0, x_2 + y_2)}
        {\pi(\abs{x - y}^2 + 4 x_2 y_2)}.
  \end{equation}
A notable feature of \eqref{eq_Symmetricpart_halfspace} is the vanishing of the tangential component.

We are going to extends formula~\eqref{eq_Symmetricpart_halfspace} to any bounded domain,
simply or non simply connected. To do this, we are going to show
that~\eqref{eq_Symmetricpart_halfspace} holds in a disk. From there, we are going to show 
that we can conformally transform $D_0$ to a disk, and find back similar estimates near
$\partial D_0$, that is: near the boundary of $\domain$ that does not correspond to islands.
The estimate would also hold in near every connected component of $\partial\domain$,
after transformation via a conformal map of the form $z\in\mathbb{C}\setminus\{0\}\mapsto 1/z$.

\begin{proposition}
\label{proposition_symmetric_gradient_estimate_boundary}
If \(\delta > 0\), there exists a constant \(C > 0\) such that if 
\(\dist (x, \partial D) + \dist (y, \partial \domain) + \abs{x - y} \le \delta\),
then 
\begin{equation*}
    \Bigg|
          \gradient\greenlaplace[\domain](x,y)
        +
          \gradient\greenlaplace[\domain](y,x)
   - \frac
            {
                x - P_{\partial \domain} (x) 
              + 
                y - P_{\partial \domain} (y)}
            {
              \pi 
              (
                  \abs{x - y}^2 
                + 
                  4 
                  \dist (x, \partial \domain) 
                  \dist (y, \partial \domain)
              )
            }
     \Bigg|
  \le
    C
  .
\end{equation*}
\end{proposition}

\begin{proof}
We start by observing the Green function of the Laplacian on the disk \(\mathbb{D}^2 \subset \plane\), 
which is defined for each \(\Tilde{x}, \Tilde{y} \in \mathbb{D}^2\) by
  \begin{equation*} 
      \greenlaplace[\mathbb{D}^2](\Tilde{x}, \Tilde{y}) 
    =   
      \frac{1}{4\pi}
      \ln 
        \,
        \biggl(
            1 
          + 
            \frac
              {(1 - \abs{\Tilde{x}}^2)(1 - \abs{\Tilde{y}}^2)}
              {\abs{\Tilde{x} - \Tilde{y}}^2}
        \biggr).
  \end{equation*}
The gradient of this function, for fixed \(\Tilde{y} \in\mathbb{D}^2\), is then given by
  \begin{equation*} 
      \gradient \greenlaplace[\mathbb{D}^2] (\Tilde{x}, \Tilde{y})
    = 
      \frac
        {-1}
        {2\pi (\abs{\Tilde{x} - \Tilde{y}}^2 + (1 - \abs{\Tilde{x}}^2)(1 - \abs{\Tilde{y}}^2)} 
      \Bigl(
          \Tilde{x} (1 - \abs{\Tilde{y}}^2) 
    +
          \frac
            {(1 - \abs{\Tilde{y}}^2) (1 - \abs{\Tilde{x}}^2)}
            {\abs{\Tilde{x} - \Tilde{y}}^2}
          (\Tilde{x} - \Tilde{y})
      \Bigr)
    ,
  \end{equation*}
and thus we have 
  \begin{equation*}
        \gradient\greenlaplace[\mathbb{D}^2](\Tilde{x}, \Tilde{y})
      +
        \gradient\greenlaplace[\mathbb{D}^2](\Tilde{y}, \Tilde{x})
    = -
      \frac
        {\Tilde{x} (1 - \abs{\Tilde{y}}^2) + \Tilde{y} (1 - \abs{\Tilde{x}^2})}
        {
          2\pi
          (
              \abs{\Tilde{x} - \Tilde{y}}^2 
            + 
              (1 - \abs{\Tilde{x}}^2)
              (1 - \abs{\Tilde{y}}^2)
          )
        }.
  \end{equation*}
We observe that 
\begin{equation*}
\begin{split}
    (1 - \abs{\Tilde{x}}^2)(1 - \abs{\Tilde{y}}^2) 
  &= 
    (1 - \abs{\Tilde{x}})
    (1 - \abs{\Tilde{y}})
    \bigl(2 -(1 - \abs{\Tilde{x}})\bigr)
    \bigl(2 - (1 - \abs{\Tilde{y}})\bigr)\\
  &=
    4 (1 - \abs{\Tilde{x}})(1 - \abs{\Tilde{y}}) 
    + 
      O \bigl((1 - \abs{\Tilde{x}})^2 + (1 - \abs{\Tilde{y}})^2\bigr) 
\end{split}
\end{equation*}
and 
\begin{equation*}
\begin{split}
      \Tilde{x} 
      (1 - \abs{\Tilde{y}}^2) 
    + 
      \Tilde{y} 
      (1 - \abs{\Tilde{x}^2})
  &=  (\Tilde{x} - \Tilde{y}) (\abs{\Tilde{y}}^2 - \abs{\Tilde{x}}^2)
    +2 
    \Bigl(
      \Tilde{x} -\frac{\Tilde{x}}{\abs{\Tilde{x}}}
      +
      \Tilde{y} - \frac{\Tilde{y}}{\abs{\Tilde{y}}}
      \Bigr)\\
&
\qquad
   + 
      \Tilde{x} 
      (\abs{\Tilde{x}} - 1)^2
      \Bigl(1 + \frac{2}{\abs{\Tilde{x}}}\Bigr)
    + 
      \Tilde{y} 
        (\abs{\Tilde{y}} - 1)^2
        \Bigl(1 + \frac{2}{\abs{\Tilde{y}}}\Bigr)
  \\
  & =
     2 
    \Bigl(
      \Tilde{x} -\frac{\Tilde{x}}{\abs{\Tilde{x}}}
    + \Tilde{y} - \frac{\Tilde{y}}{\abs{\Tilde{y}}}
      \Bigr)
      + O \bigl(\abs{x - y}^2 + (1 - \abs{\Tilde{x}})^2 + (1 - \abs{\Tilde{y}})^2\bigr).
      \end{split}
\end{equation*}
It thus follows that when \(\abs{x - y}^2 + (1 - \abs{\Tilde{x}})^2 + (1 - \abs{\Tilde{y}})^2\) is small enough,
\begin{equation}
\label{eq_symmetricpart_disk}
    \strabs{
          \gradient\greenlaplace[\mathbb{D}^2](\Tilde{x}, \Tilde{y})
        +
          \gradient\greenlaplace[\mathbb{D}^2](\Tilde{y}, \Tilde{x})
      -
        \frac
          {\Tilde{x} - \frac{\Tilde{x}}{\abs{\Tilde{x}}} + \Tilde{y} - \frac{\Tilde{y}}{\abs{\Tilde{y}}}}
          {\pi(\abs{\Tilde{x} - \Tilde{y}}^2 + 4 (1 - \abs{\Tilde{x}})(1 - \abs{\Tilde{y}}))}
          }
  \le 
    \C
    .
\end{equation}

By the classical Riemann mapping theorem (see for example \cite{Krantz_2006}*{theorems~4.0.1 and 5.2.1}),
there exists a map \(\Phi \in C^2 (\Bar{\domain}_0, \Bar{\mathbb{D}^2})\) which is a diffeomorphism up to the boundary and which is conformal map.
For each \(i \in \{1, \dotsc, \nbislands\}\), we have \(\Phi (\island_i) \cap \partial \mathbb{D}^2 = \emptyset\).

We define the function \(\breve{G}_D\) for each $x,y\in\domain$ as
\begin{equation*}
    \breve{G}_D (x, y) 
  \defeq 
    G_{\mathbb{D}^2} (\Phi (x), \Phi (y)) .
\end{equation*}
We compute
\begin{equation*}
  \nabla \breve{G}_D (x, y)
 = 
  (D\Phi(x))^*[\nabla G_{\mathbb{D}^2} (\Phi (x), \Phi (y))],
\end{equation*}
and thus 
\begin{multline*}
  \nabla \breve{G}_D (x, y)
+ \nabla \breve{G}_D (y, x)
= (D\Phi(x))^*[\nabla G_{\mathbb{D}^2} (\Phi (x), \Phi (y))
+ \nabla G_{\mathbb{D}^2} (\Phi (y), \Phi (x))]\\
+ (D \Phi (y) - D\Phi (x))^* [\nabla G_{\mathbb{D}^2} (\Phi (y), \Phi (x))].
\end{multline*}
We observe that when \(\abs{x - y} + \dist (x, \partial \domain_0) + \dist (y, \partial \domain_0) \to 0\), we have 
\begin{align*}
    \Phi (x) - \frac{\Phi (x)}{\abs{\Phi(x)}}
 & =      
    D \Phi (x) [x - P_{\partial D} (x)] 
  + 
    O (\dist (x, \partial \domain)^2),
    \\
      \Phi (y) - \frac{\Phi (y)}{\abs{\Phi(y)}}
 &=      
    D \Phi (x) [y - P_{\partial D} (y)] 
  + 
    O (\dist (y, \partial \domain)^2 + \abs{x - y}^2),\\
 \abs{\Phi (x) - \Phi (y)}^2 & = \abs{D \Phi (x)}^2 \abs{x - y}^2
 + O (\abs{x - y}^3),\\
 (1 - \abs{\Phi (x)}) &= \frac{\dist (x, \partial D)}{\abs{D \Phi (x)}}  + O ((\dist x, \partial \domain)^2),\\
 (1 - \abs{\Phi (y)}) &= \frac{\dist (y, \partial D)}{\abs{D \Phi (x)}}  + O ((\dist y, \partial \domain)^2 + \abs{x - y}^2),
\end{align*}
from which we deduce that 
\begin{multline}
\label{eq_cddabbe3c1}
\biggabs{
 (D\Phi(x))^*[\nabla G_{\mathbb{D}^2} (\Phi (x), \Phi (y))
+ \nabla G_{\mathbb{D}^2} (\Phi (y), \Phi (x))]
        \\
        - \frac
            {
                x - P_{\partial \domain} (x) 
              + 
                y - P_{\partial \domain} (y)}
            {
              \pi 
              (
                  \abs{x - y}^2 
                + 
                  4 
                  \dist (x, \partial \domain) 
                  \dist (y, \partial \domain)
              )
            }
}
\le \C.
\end{multline}
We also have immediately
\begin{equation}
\label{eq_3aae14d5e7}
 \abs{
 (D \Phi (y) - D\Phi (x))^* [\nabla G_{\mathbb{D}^2} (\Phi (y), \Phi (x))]
 }
 \le \C.
\end{equation}
Now we draw the link between the transport $\Breve{G}_D$ and the Green's function $G_D$
we target.

Since the map \(\Phi\) is conformal, the function \(G_{\mathbb{D}^2} (\cdot, y)\) is harmonic in \(\mathbb{D}^2 \setminus \{y\}\) and there exists \(\delta > 0\) such that if \(\dist (y, \partial \domain_0) \le \delta\), then \(\Theta (\cdot, y)
\defeq G_D (\cdot, y) - \Breve{G}_D (\cdot, y)\) is bounded uniformly in a neighborhood of \(\bigcup_{i = 1}^\nbislands \island_i\); this implies that for every \(x, y \in \domain\) such that \(\dist (x, \partial \domain) + \dist (y, \partial \domain) \le \delta\), we have 
\begin{equation}
\label{eq_1636b035e6}
 \abs{\nabla \Theta (x, y)} \le \C.
\end{equation}

The conclusion in when \(\dist (x, \partial \domain_0) + \dist (y, \partial \domain_0) + \abs{x - y} \le \delta\) follows by combining the estimates \eqref{eq_cddabbe3c1}, \eqref{eq_3aae14d5e7} and \eqref{eq_1636b035e6}.
The other components \(I_1, \dotsc, I_{\nbislands}\) of the boundary can be reduced to this case by a suitable adaptation of the conformal mapping $z\in\mathbb{C}\setminus\{0\}\mapsto 1/z$. 
\end{proof}

\section{Vortex estimates}
\label{section_concentration}

In this section we derive several estimates on the vorticity that govern the concentration of the vorticity. 

In order to control the shape of the vortex, we will recurrently rely on the Lorentz norm \cite{Lorentz_1950} of a vorticity  \(\vortex : \domain \to \reals\) which will be defined as 
\begin{multline}
  \label{eq_def_Lorentznorm}
      \lorentznorm{\vortex}
      \defeq 
      \sup 
      \biggl\{ 
        \int_{\plane} \Bigl(\ln \frac{1}{\abs{x}}\Bigr)_+ \tilde{\vortex} (x)  \dif x
        \st \text{\(\tilde{\vortex} : \plane \to \reals\) and for every \(\lambda > 0\) }, \\
        \abs{\{x \in \domain \st \abs{\vortex (x)} > \lambda \}}        = \abs{\{x \in \plane \st \abs{\tilde{\vortex} (x)} > \lambda \}}
     \biggr\}.
\end{multline}
We use the Lebesgue measure in the definition, despite the fact that the flow transports the measure with density \(b\) of the potential vorticity \(\vortex/b\).  

By the Hardy--Littlewood rearrangement inequality (see for example \cite{Lieb_Loss_2001}*{theorem 3.4}), the  supremum in \eqref{eq_def_Lorentznorm} is actually reached by the
radially symmetric nonincreasing rearrangement \(\vortex^\star\) of \(\vortex\), whose superlevel sets are balls centered on \(0\).

\subsection{Stream function estimate}
We first show how the Lorentz norm can be used to obtain a bound on the stream function.

\begin{proposition}%
[Boundedness of the stream function]%
\label{controlCondition}
There exists a constant $C>0$ that depends only on $\domain$ and $\depth$,
such that
for every non-negative function \(\vortex \in L^p (\domain)\) and every \(\rho > 0\),
we have
  \begin{multline*} 
      \norm{\mathcal{K}_b[\vortex]}{L^\infty (\domain)}
    \leq 
    \frac{1}{2 \pi}\int_{D}  \ln \frac{1}{\max (\rho,\abs{x -y})} \,\vortex (y)\, b(y) \dif y
      + 
        \frac
          {(\sup_{\domain} b) \,\rho^2 \lorentznorm{\vortex(\rho \cdot)}}
          {2 \pi}
      + C\int_\domain\vortex.
  \end{multline*}

\end{proposition}
\begin{proof}
By writing the function \(\mathcal{K}_b[\vortex]\) in terms of integral kernels of \cref{velocityIntegralExpansion}, we have for each \(x \in \domain\),
  \begin{equation*} 
    \begin{split}
          0
      \le 
            \mathcal{K}_b[\vortex](x) 
      = 
        &
        \int_{\domain}
          \greenlaplace(x,y)\,
          \vortex(y)\,
          \depth(y)\dif y
        + 
        \int_{\domain}
          R_b(x,y)\,
          \vortex(y)\dif y  \\
        &\qquad +
        \int_{\domain}
          \greenlaplace(x,y)\,
          \vortex(y)
          \Bigl(\sqrt{b (x)} - \sqrt{b (y)}\Bigr)
          \sqrt{b (y)}
          \dif y 
  .
    \end{split}
  \end{equation*}
By \cref{velocityIntegralExpansion}, the function $R_b$ is uniformly bounded.
According to~\cref{propositionGreenLogBound} and the fact that the function 
$\sqrt{b}$ is Lipschitz-continuous, the term
	\[ \greenlaplace(x,y)\,
          \Bigl(\sqrt{b (x)} - \sqrt{b (y)}\Bigr)
          \sqrt{b (y)} \]
is uniformly bounded as $(x,y)\in\domain\times\domain$.
Moreover,
using the direct estimate of \cref{propositionGreenLogBound}, we obtain for every \(x \in \domain\)
  \begin{equation*}
              \int_{\domain}
          \greenlaplace(x,y)\,
          \vortex(y)\,
          \depth(y)\dif y 
   \leq         
        \frac{1}{2\pi}
        \int_{\domain}
          \ln\frac{\diam \domain}{\abs{x - y}}
          \vortex(y)
          \ 
          \depth(y)
          \dif y .
  \end{equation*}
Since for every \(x, y \in \plane\) such that \(x \ne y\),
\[
\ln\frac{\diam \domain}{\abs{x - y}}
= \ln \frac{\diam \domain}{\max (\rho, \abs{x - y})} + \biggl(\ln \frac{\rho}{\abs{y - x}}\biggr)_+,
\]
we have then for every \(x \in \domain\),
  \begin{multline*}
      \mathcal{K}_b{[\vortex]}(x) 
    \leq 
        \frac{1}{2\pi}
        \int_{D} \ln\frac{1}{\max (\rho, \abs{x - y})}\,\vortex(y)\, b(y) \dif y\\
      + 
        \frac{(\sup_\domain\depth)\, \rho^2}{2\pi}
        \int_{x + \rho z \in \domain}
          \positivepart{\ln\frac{1}{|z|}}
          \ \vortex(x + \rho z)
          \dif z
      + C\int_\domain \vortex ,
  \end{multline*}
for some constant $C>0$ that depends only on $\domain$ and $\depth$.
The conclusion now follows from the definition of the Lorentz norm \eqref{eq_def_Lorentznorm}.
\end{proof}

\subsection{Energy concentration estimate}
The following proposition gives estimates on the kinetic energy outside a ball.

\begin{proposition}
  \label{proposition_energy_concentration}
There exists a constant $C>0$ that depends only on $\domain$ and $\depth$,
  such that
  for every non-negative function \(\vortex \in L^p (\domain)\) and every \(R, r, \rho > 0 \) such that \(R \ge r + \rho\) and \(r \ge \rho\),
  we have
\begin{multline}
\int_{\domain \setminus B (a, R)} 
\abs{\velocity}^2 
\le 
C \biggl(\abs{\Gamma} \ln \frac{1}{\rho} \int_{\domain \setminus B (a, r)} \vortex
+ \abs{\Gamma}^2 \ln \frac{1}{R - r} + \rho^2 \lorentznorm{\vortex(\rho \cdot)}\abs{\Gamma} + \norm{\Gamma}{}^2
\biggr).
\end{multline}
\end{proposition}
Here and in the sequel, we use the notation
\begin{equation}
\label{eq_norm_Gamma}
 \norm{\Gamma}{} \defeq \abs{\Gamma} + \sum_{i = 1}^\nbislands \abs{\Gamma_i}.
\end{equation}
\begin{proof}[Proof of \cref{proposition_energy_concentration}]
We define the function \(\phi \defeq \mathcal{K}_b [\vortex] + 2 \sum_{i = 1}^{\nbislands} \Gamma_i \psi_i\).

For every \(x \in \domain \setminus B (a, R)\), we have if \(y \in B (a, r)\), \(\abs{x - y} \ge \abs{x - a} - \abs{y - a} \ge R - r \ge \rho\),
and thus by \cref{controlCondition}, 
\begin{equation}\label{eq_ov3aeBoo8uage5shu}
\begin{split}
  \phi (x)
  &\le 
  \frac{1}{2 \pi} \ln \frac{1}{\rho} \int_{\domain \setminus B (a, r)} \vortex b
  +
  \frac{1}{2 \pi} \ln \frac{1}{R - r}
  \int_{B (a, r)} \vortex b
  + \C \bigl(\rho^2 \lorentznorm{\vortex(\rho \cdot)} + \norm{\Gamma}{}\bigr)
  \\
  & \le 
  \lambda \defeq
  \Cl{cst} \Bigl( 
  \ln \frac{1}{\rho} \int_{\domain \setminus B (a, r)} \vortex
  +\abs{\Gamma} \ln \frac{1}{R - r} 
  + \rho^2 \lorentznorm{\vortex(\rho \cdot)} + \norm{\Gamma}{}\Bigr).
\end{split}
\end{equation}
We have by the representation formula for the velocity field \eqref{eq_velocity_reconstruction}, as in \eqref{eq_ohHiis8aaYae8yul4} in the proof of \cref{proposition_Energy_Formula}, 
\begin{equation}
\begin{split}
  \label{eq_quuSeiC8she0thoh1}
  \int_{\domain \setminus B (a, R)} 
  \abs{\velocity}^2
  &\le    \int_{\phi^{-1} ((0, \lambda))} 
  \abs{\velocity}^2\\
  &=
  \frac{1}{2}
  \int_{\phi^{-1} ((0, \lambda))} 
  \frac
  {\nabla \mathcal{K}_b[\vortex] 
    \cdot
    \textstyle \nabla \phi
  }
  {b}
  +
  \sum_{i, j = 1}^{\nbislands}
  \frac{\Gamma_i\, \Gamma_j}{2}
  \int_{\phi^{-1} ((0, \lambda))} 
  \frac
  {
    \nabla \psi_i \cdot \nabla \psi_j
  }
  {b}.
  \end{split}
\end{equation}
We have then,  
\begin{multline}
  \label{eq_Aebeopo9eelohgohV}
\frac{1}{2}
\int_{\phi^{-1} ((0, \lambda))} 
\frac
{\nabla \mathcal{K}_b[\vortex] 
  \cdot
  \nabla \phi
}
{b}
= \frac{1}{2} \int_{\domain} \frac{\nabla \mathcal{K}_b[\vortex]
  \cdot \nabla \max (\phi, \lambda)}{b}
= \frac{1}{2} \int_{\domain} \vortex \max (\phi,\lambda)
\le \frac{\abs{\Gamma}{\lambda}}{2}.
\end{multline}
On the other hand, we have   
\begin{equation}
  \label{eq_ohbeithah1Oash2bi}
\sum_{i, j = 1}^{\nbislands}
\frac{\Gamma_i\, \Gamma_j}{2}
\int_{\phi^{-1} ((0, \lambda))} 
\frac
{
  \nabla \psi_i \cdot \nabla \psi_j
}
{b}
\le \C \norm{\Gamma}{}^2.
\end{equation}
The conclusion follows from the combination of  \eqref{eq_quuSeiC8she0thoh1}, \eqref{eq_Aebeopo9eelohgohV} and \eqref{eq_ohbeithah1Oash2bi}, in view of the definition of \(\lambda\) in \eqref{eq_ov3aeBoo8uage5shu}.
\end{proof}

\subsection{Vortex concentration}
The next estimate shows that there is a characteristic radius \(\rho\) defined in terms of conserved quantities such that if the Lorentz norm at the scale \(\rho\) remains bounded, then the vorticity is concentrated in a region of radius comparable to \(\rho\).

\begin{proposition}%
[Concentration estimate]%
\label{concentrationEstimates}
There exists a constant \(C>0\) that depends only on \(\domain,\depth\),  such that
for all \(R>1\), we have
  \begin{equation*} 
      \inf_{a \in \domain} 
      \int_{\domain\setminus B(a, R\spacetypscale)}
        \vortex 
    \leq 
      \frac{C}{\ln(R)}
      \left( \rho^2 
          \lorentznorm{\vortex(\rho \cdot)} 
        + 
          \frac{\norm{\Gamma}{}^2}{\abs{\Gamma}} 
      \right),
\end{equation*}
where  
  \begin{equation}\label{defTypscale}
      \spacetypscale 
    \defeq 
      \exp
        \bigg( 
          -\frac
            {4\pi\,\energy}
            {\vortexstrength\ \totalvorticity} 
        \biggr)
    .
  \end{equation} 
\end{proposition}

Here and in the sequel, we reserve the symbol $\rho$ to refer to the quantity
defined in~\eqref{defTypscale}.

The proof of \cref{concentrationEstimates},
follows ideas introduced by Turkington~\cite{TurkingtonEvolution}
and Turkington~\&~Friedmann~\cite{Friedman_Turkington_1981} for the Euler equations. It was also used in chapters~5~and~6
in the study of steady solutions of the lake equations by energy maximization~\citelist{\cite{Dekeyser1}\cite{Dekeyser2}}.


\begin{proof}%
[Proof of \cref{concentrationEstimates}]
We define
the set 
  \begin{equation} 
  \label{eq_zaasi0Mi2e}
      A 
    \defeq 
      \biggl\{ 
          x \in \domain 
        \st 
            \psi (x)
          \geq 
            \frac{1}{\Gamma}
            \int_{\domain} 
              \psi
              \,
              \vortex 
      \biggr\} ,
  \end{equation}
in terms of the stream function \(\psi \defeq \mathcal{K}_b [\vortex] + \sum_{i = 1}^k \Gamma_i \psi_i\).
We observe that by definition of the vortex circulation \(\Gamma\) in \eqref{eqVortexCirculation}, the set \(A\) is not empty.

By \cref{velocityIntegralExpansion} and \cref{propositionGreenLogBound}, since by assumption the function \(b\) is Lipschitz-continuous, we have for every \(x \in \domain\),
\begin{equation*}
\begin{split}
    \mathcal{K}_b [\vortex] (x)
    &
  \le 
      \frac{1}{2 \pi}
      \int_{\domain} 
        \ln \frac{\diam \domain}{\abs{x - y}}
        \sqrt{ b (x)\, b (y)}
        \dif y
    +
      \int_{\domain}
        R_b (x, y) 
        \,
        \vortex (y)
        \dif y
        \\
        &
  \le
      \frac{1}{2 \pi}
      \ln \frac{1}{\rho}
      \int_{\domain} 
        \vortex\, b
    +
      \frac{1}{2 \pi}
      \int_{\domain}
        \ln \frac{\rho}{\abs{x - y}} 
        \,
        \vortex (y)
        \,
        b (y)
        \dif y
    +
      \C \abs{\Gamma},
\end{split}    
\end{equation*}
and thus by definition of \(\psi\) and of \(\norm{\Gamma}{}\) in \eqref{eq_norm_Gamma},
\begin{equation}
\label{eq_OWJmNTBhZG}
    \psi (x)
 \le
      \frac{1}{2 \pi}
        \ln \frac{1}{\rho}
        \int_{\domain} 
          \vortex\, b
    +
      \frac{1}{2 \pi}
      \int_{\domain}
        \ln \frac{\rho}{\abs{x - y}} 
        \,
        \vortex (y)\,
        b (y)
        \dif y
    +
      \C 
      \norm{\Gamma}{}.
\end{equation}

On the other hand,
setting \(\rho\) to the value given by \eqref{defTypscale}, we obtain for each \(x \in A\), 
in view of \cref{proposition_Energy_Formula}
and the definition of the set \(A\) by \eqref{eq_zaasi0Mi2e}
\begin{equation}
\label{eq_N2U1YTc0Nj}
\begin{split}
    \psi (x) 
  \ge 
      \frac
        {2 E}
        {\Gamma} 
    - 
      \sum_{i = 1}^{\nbislands}
        \frac{\Gamma_i}{\Gamma}
        \int_{\domain}
          \vortex
          \,
          \psi_i
    - \sum_{i, j = 1}^{\nbislands}
        \frac{\Gamma_i \, \Gamma_j}{\Gamma}
        \int_{\domain}
          \nabla \psi_i \cdot \nabla \psi_j
   \ge 
        \frac{1}{2 \pi} 
        \ln \frac{1}{\rho} 
        \int_{\domain} 
        \vortex 
        \,
          b
      -
        \C 
        \frac
          {\norm{\Gamma}{}^2}
          {\abs{\Gamma}}
        .
\end{split}
\end{equation}
The combination of \eqref{eq_OWJmNTBhZG} and \eqref{eq_N2U1YTc0Nj}, shows  that for each \(x \in A\),
  \begin{equation}
   \label{eq_hgiwuSZNbv8F}
        \frac{1}{2\pi}
        \int_\domain
          \ln
            \frac
              {\abs{x - y}}
              {\spacetypscale}
          \,
          \vortex(y)
          \,
          \depth(y)
          \dif y      
    \leq 
        \C 
        \frac
          {\norm{\Gamma}{}^2}
          {\abs{\Gamma}} .
  \end{equation}

In order to conclude, we start from the inequality 
  \begin{equation}
   \label{eq_uanlC6UnYU0}
      \ln(R)
      \,
      \frac
        {\inf_\domain\depth}
        {4\pi}
      \int_{\domain\setminus B(x,R\spacetypscale)}
        \vortex(y)
        \dif y
    \le
        \frac{1}{4\pi}
        \int_{\domain\setminus B(x,R\spacetypscale)}
          \ln
            \frac
              {\abs{x - y}}
              {\spacetypscale}
              \,
          \vortex(y)
          \,
          \depth(y)
          \dif y.
  \end{equation}
We also observe that
  \begin{equation}
   \label{eq_s90rTemdGNKD}
   \begin{split}
    \frac
      {1}
      {4\pi}
      \int_{\domain\cap B(x,R\spacetypscale)}
        \ln \frac{\spacetypscale}{\abs{x - y}}
        \,
        \vortex(y)
        \,
        \depth(y)
        \dif y
  &\leq
    \frac{\rho^2 \sup_\domain\depth}{4\pi}
    \int_{\domain}
      \positivepart{\ln\frac{1}{\abs{z}}}
      \,
      \vortex(x + \rho z)
      \dif z \\
   &\leq 
      \Cl{cst_b1a0d0f4}
      \, 
      \rho^2
  \,
      \lorentznorm{\vortex(\rho \cdot)} ,
  \end{split}
  \end{equation}
  in view of the definition of the Lorentz norm \eqref{eq_def_Lorentznorm}.
Therefore, in view of \eqref{eq_hgiwuSZNbv8F}, \eqref{eq_uanlC6UnYU0} and \eqref{eq_s90rTemdGNKD}, we have 
for some constant \(\Cl{cst_cf81528} >0\) and for each \(x\in A\) and \(R>1\):
  \begin{equation*} 
      \int_{\domain\setminus B(x,R\spacetypscale)}
        \vortex
    \leq 
      \frac{\Cr{cst_cf81528}}{\ln(R)}
      \left( 
          \rho^2 
          \lorentznorm{\vortex(\rho \cdot)} 
        + 
          \frac
            {\norm{\Gamma}{}^2}
            {\abs{\vortexstrength}}
      \right) 
      .
    \qedhere
  \end{equation*}
\end{proof}

\subsection{Boundary repulsion}

The next estimate shows that the vorticity \(\vortex\) cannot be concentrated too much in a neighborhood of the boundary when \(\rho\) is small.

\begin{proposition}
\label{propositionConfinementEstimate}
If \(\rho\) satisfies \eqref{defTypscale}, then 
\begin{equation*}
  \int_{\domain} 
    \vortex (x) 
    \ln 
    \frac{1}{\rho + C \dist (x, \partial \domain)}
    \dif x
  \le
    C 
    \left(\rho^2 \lorentznorm{\vortex (\rho \cdot)} + \frac{\norm{\Gamma}{}^2}{\abs{\Gamma}} \right).
\end{equation*}
\end{proposition}

\begin{proof}
By the energy identity of \cref{proposition_Energy_Formula},
the decomposition of the Green function of \cref{velocityIntegralExpansion} and the upper bound on the Green function of the Laplacian of 
\cref{proposition_Green_D_upper},
we have 
\begin{equation}
\label{eq_M2EwNjhkYj}
    E  =
      \frac{1}{2}
      \int_{\domain} 
        \vortex 
        \,
        \mathcal{K}_b[\vortex] 
    + 
      \sum_{i = 1}^{\nbislands}\Gamma_i
      \int_{\domain}
        \psi_i 
        \vortex
    + 
      \sum_{i, j = 1}^{\nbislands} 
        \frac
          {\Gamma_i \, \Gamma_j}
          {2}
        \int_{\domain} \nabla \psi_i \cdot \nabla \psi_j
   \leq I +  \Cl{cst_NTYwNjgxDZ}  \norm{\Gamma}{}^2,
\end{equation}
where we have set
	\begin{equation*}
	 I \defeq
      \frac{1}{8 \pi}
      \iint\limits_{\domain \times \domain}
        \ln 
          \biggl(
              1 
            + 
              \Cl{cst_NTYwNjgxZD} 
              \frac
                {\dist (x, \partial \domain) \dist (y, \partial \domain)}
                {\abs{x - y}^2} 
          \biggr)
        \vortex (x)
        \,
        \vortex (y)
        \,
        \sqrt{b (x)\, b (y)}
        \dif x
        \dif y 
    \end{equation*}
which can be bounded as
\begin{multline}
\label{eq_OTFlM2YyZT}
 I\le 
      \frac{\sup_{\domain} b}{8 \pi}
      \iint\limits_{\domain \times \domain}
        \ln 
          \biggl(
              \rho^2 
            + 
              \Cr{cst_NTYwNjgxZD} 
              \frac
                {\rho^2 \dist (x, \partial \domain) \dist (y, \partial \domain)}
                {\abs{x - y}^2} 
          \biggr)
          \,
          \vortex (x)
          \,
          \vortex (y)
        \dif x
        \dif y
  \\ +
      \frac{\ln \frac{1}{\rho}}{4 \pi}
      \iint_{\domain \times \domain}
        \vortex (x)
        \,
        \vortex (y)
        \sqrt{b (x) b (y)}
        \dif x
        \dif y.
\end{multline}
The integral of the second term on the right-hand side of \eqref{eq_OTFlM2YyZT} can be bounded by the Cauchy--Schwarz inequality, as
\begin{multline}
\label{eq_ODdjMjEzNT}
  \iint\limits_{\domain \times \domain}
      \vortex (x)
      \,
      \vortex (y)
      \,
      \sqrt{b (x) \, b (y)}
      \dif x
      \dif y
 \\[-1em]
  \le 
    \Biggl(
    \,
      \iint\limits_{\domain \times \domain}
        \vortex (x)
        \,
        \vortex (y)
        \,
        b (x)
        \dif x
        \dif y
    \Biggr)^{\frac{1}{2}}
    \Biggl(
    \,
      \iint\limits_{\domain \times \domain}
        \vortex (x)
        \,
        \vortex (y)
        \,
        b (y)
        \dif x
        \dif y
    \Biggr)^{\frac{1}{2}}
 	= \totalvorticity \, \Gamma,
\end{multline}
which leads to
\begin{equation} 
\label{eq_iQuoTajaengainooWesoh5uu}
\frac{\ln \frac{1}{\rho}}{4 \pi}
      \iint_{\domain \times \domain}
        \vortex (x)
        \,
        \vortex (y)
        \sqrt{b (x) b (y)}
        \dif x
        \dif y
 \leq \frac{\totalvorticity\Gamma}{4\pi}\ln\frac{1}{\rho}
 	= E,
\end{equation}
by definition of $\rho$ (see \eqref{defTypscale}).
From \eqref{eq_M2EwNjhkYj}, \eqref{eq_OTFlM2YyZT} and \eqref{eq_iQuoTajaengainooWesoh5uu} we infer that 
	\begin{equation}
		\label{eq_OTFlM2YyTZ}	
	-\Cr{cst_NTYwNjgxDZ} \Vert\Gamma\Vert^2 \leq \frac{C}{8 \pi}
      \iint\limits_{\domain \times \domain}
        \ln 
          \biggl(
              \rho^2 
            + 
              \Cr{cst_NTYwNjgxZD} 
              \frac
                {\rho^2 \dist (x, \partial \domain) \dist (y, \partial \domain)}
                {\abs{x - y}^2} 
          \biggr)
          \,
          \vortex (x)
          \,
          \vortex (y)
        \dif x
        \dif y .\end{equation}
We split the integral of the right-hand side of \eqref{eq_OTFlM2YyTZ} in two regions, depending on whether \(\abs{x - y} \le \rho\). For the first one we have, since the domain \(\domain\) is bounded, 
we have for every \(x, y \in \domain\),
\begin{equation}
\label{eq_vae6eit3Wah7uS9ohQuahnga}
 \abs{x - y}^2  + \Cr{cst_NTYwNjgxZD} 
            \dist (x, \partial \domain) \dist (y, \partial \domain) 
\le \Cl{cst_oiw3uxie5Rahquohgoo8woSh}
\end{equation}
and thus by \eqref{eq_vae6eit3Wah7uS9ohQuahnga} and by the definition of the Lorentz norm in \eqref{eq_def_Lorentznorm}
\begin{equation}
\label{eq_tmhwoyuHEHkh}
\begin{split}
    \iint\limits_{\substack{(x, y) \in \domain \times \domain \\ \abs{x - y} \le \rho}}
      \ln 
        \biggl(
            \rho^2 
          + 
&
            \Cr{cst_NTYwNjgxZD} 
            \frac
              {\rho^2 \dist (x, \partial \domain) \dist (y, \partial \domain)}
              {\abs{x - y}^2} 
        \biggr)
        \vortex (x)
        \,
        \vortex (y)
      \dif x
      \dif y\\[-1em]
  &\le 
    2
    \iint\limits_{\substack{(x, y) \in \domain \times \domain \\ \abs{x - y} \le \rho}}
      \ln 
        \frac{\sqrt{\Cr{cst_oiw3uxie5Rahquohgoo8woSh}} \rho}{\abs{x - y}}
        \,
        \vortex (x)
        \,
        \vortex (y)
      \dif x
      \dif y  \\
  &\le 
    \C 
    \,
    \bigl(
        \rho^2 \lorentznorm{\vortex (\rho \cdot)} 
      + 
        \abs{\Gamma}
    \bigr)
    \,
    \abs{\Gamma};
\end{split}
\end{equation}
for the second part we have 
\begin{equation}
\label{eq_Rf1EajVZ5Dd}
\begin{split}
    \iint\limits_{\substack{(x, y) \in \domain \times \domain \\ \abs{x - y} \ge \rho}}
      & \ln 
        \biggl(
            \rho^2 
          + 
            \Cr{cst_NTYwNjgxZD} 
            \frac
              {\rho^2 \dist (x, \partial \domain) \dist (y, \partial \domain)}
              {\abs{x - y}^2} 
        \biggr)
        \,
        \vortex (x)
        \,
        \vortex (y)
      \dif x
      \dif y\\
 & \le
      \iint\limits_{\domain \times \domain}
      \ln 
        \bigl(
            \rho^2 
          + 
            \Cr{cst_NTYwNjgxZD} 
              \dist (x, \partial \domain) \dist (y, \partial \domain)
        \bigr)
        \,
        \vortex (x)
        \,
        \vortex (y)
      \dif x
      \dif y\\
 & \le
      \iint\limits_{\domain \times \domain}
      \ln 
        \bigl(
            (\rho
          + 
            \sqrt{\Cr{cst_NTYwNjgxZD}}
              \dist (x, \partial \domain))
             (\rho  +  \sqrt{\Cr{cst_NTYwNjgxZD}} \dist (y, \partial \domain)
             )
        \bigr)
        \,
        \vortex (x)
        \,
        \vortex (y)
      \dif x
      \dif y\\
&\le 
      2
      \iint\limits_{\domain \times \domain}
      \ln 
        (\rho
          + 
            \sqrt{\Cr{cst_NTYwNjgxZD}}
              \dist (x, \partial \domain))
        \,
        \vortex (x)
        \,
        \vortex (y)
      \dif x
      \dif y\\
  &  =       2 \,\abs{\Gamma}
      \int_{\domain}
      \ln 
        (
            \rho + 
               \sqrt{\Cr{cst_NTYwNjgxZD}}\dist (x, \partial \domain)) 
        \,
        \vortex (x)
      \dif x.
      \end{split}
\end{equation}
By \eqref{eq_OTFlM2YyTZ}, \eqref{eq_tmhwoyuHEHkh} and \eqref{eq_Rf1EajVZ5Dd}, we deduce then that 
\begin{equation}
 \int_{\domain}
      \ln 
        \frac{1}{
            \rho + 
               \sqrt{\Cr{cst_NTYwNjgxZD}}\dist (x, \partial \domain)}
        \,
        \vortex (x)
      \dif x
      \le \C 
    \,
    \left(
        \rho^2 \lorentznorm{\vortex (\rho \cdot)} 
      + 
        \frac{\norm{\Gamma}{}^2}{\abs{\Gamma}}
    \right),
\end{equation}
and the conclusion follows.
\end{proof}

\subsection{Center of vorticity}

We define now the \emph{center of vorticity}
\begin{equation*}
    q
  \defeq
    \frac{1}{\Gamma} 
    \int_{\domain} 
      x
      \,
      \vortex (x) 
      \dif x
\end{equation*}
and we prove that concentration occurs in fact around the center of vorticity, as a consequence of \cref{concentrationEstimates}.

\begin{proposition}[Concentration around the center of vorticity]
\label{proposition_cvort_conc}
There exists constants $C, C'>0$ such that,
for all $R>1$, we have
\begin{equation*}
 \int_{\domain \setminus B (q, \rho_*(\rho, R)) } \vortex
\le \frac{C}{\ln(R)}
      \biggl( \rho^2 
          \lorentznorm{\vortex(\rho \cdot)} 
        + 
          \frac{\norm{\vortexstrength}{}^2}{\abs{\vortexstrength}}
      \biggr),
\end{equation*}
with  \(\rho\) defined in \eqref{defTypscale} and 
\begin{equation*}
 \rho_* (\rho, R) =   
    R \rho 
    +  
    \frac{C'}{\ln(R)}
      \biggl( \frac{\rho^2 
          \lorentznorm{\vortex(\rho \cdot)}}{\abs{\Gamma}}
        + 
          \frac{\norm{\Gamma}{}^2}{\abs{\Gamma}^2}
      \biggr).
\end{equation*}
\end{proposition}
\begin{proof}
Let $R>1$ be any number.
By \cref{concentrationEstimates}, there exists a constant $\Cl{cst_Shoo5aeth5}$
(independent on $R>1$) and some point \(a \in \domain\) such that
\begin{equation}
\label{eq_QiDpt8h14RQe}
 \int_{\domain \setminus B (a, R \rho)} \vortex 
\le
        \frac{\Cr{cst_Shoo5aeth5}}{\ln(R)}
      \biggl( \rho^2 
          \lorentznorm{\vortex(\rho \cdot)} 
        + 
          \frac{\norm{\Gamma}{}^2}{\abs{\Gamma}}
      \biggr).
\end{equation}
We now compute
\begin{equation}
\label{eq_Ahdo4meaPo}
\begin{split}
    \abs{q - a}
  =
    \frac{1}{\abs{\Gamma}}
    \Bigabs{\int_{\domain}      
      (x - a) \,\vortex (x) 
      \dif x}
  &\le
      \frac{R \rho}{|\Gamma|} \int_{B (a, R \rho)} \vortex
    + 
      \frac{\diam \domain}{|\Gamma|}
      \int_{\domain \setminus B (a, R \rho)}
        \vortex\\
  &\le 
    R \rho 
    +  
    \frac{\Cl{cst_oe9ahB3bai}}{\ln(R)}
      \biggl( \frac{\rho^2 
          \lorentznorm{\vortex(\rho \cdot)}}{\abs{\Gamma}}
        + 
          \frac{\norm{\Gamma}{}^2}{\abs{\Gamma}^2}
      \biggr),
      \end{split}
\end{equation}
for some other constant $\Cr{cst_oe9ahB3bai}>0$ independent of $R>1$ and on $\vortex$.
The conclusion follows from \eqref{eq_Ahdo4meaPo} and \eqref{eq_QiDpt8h14RQe}.
\end{proof}

As a consequence of \cref{propositionConfinementEstimate} we estimate the distance between the center of vorticity \(q\) to the boundary \(\partial \domain\).

\begin{proposition}[Confinement of the center of vorticity]
\label{propositionConfinementCenter}
There exists a constant $C>0$ such that for all \(R > 1\),
\begin{equation*}
    C 
    \Biggl( 
        \frac
          {\rho^2 \lorentznorm{\vortex(\rho \cdot)}}
          {\abs{\Gamma}{}} 
      +
        \frac{\norm{\Gamma}{}^2}{\abs{\Gamma}^2}
    \Biggr)
  \Biggl(
      \frac
        {1}
        { 
          \ln 
              \frac
                  {1}
                  {\rho + C(\dist (q, \partial \domain) + \rho_* (R, \rho))}
        }
    + \frac{1}{\ln R}\Biggr) \ge 1,
\end{equation*}
where \(\rho\) defined in \eqref{defTypscale} and \(\rho_* (R, \rho)\) is defined in \cref{proposition_cvort_conc}.
\end{proposition}

\begin{proof}
[Proof of \cref{propositionConfinementCenter}]
By \cref{proposition_cvort_conc}, we have
\begin{equation}
\label{eq_NDEzOTY2Zj}
 \int_{\domain \setminus B (q, \rho_*(R, \rho)) } \vortex
\le \frac{\C}{\ln(R)}
      \biggl( \rho^2 
          \lorentznorm{\vortex(\rho \cdot)} 
        + 
          \frac{\norm{\Gamma}{}^2}{\abs{\vortexstrength}}
      \biggr)
      .
\end{equation}
On the other hand by \cref{propositionConfinementEstimate},
we have 
\begin{equation}
\label{eq_YjQzY2RlOT}
\begin{split}
    \int_{B (q, \rho_* (R, \rho)) } \vortex
  &\le 
    \frac
      {1}
      {
        \ln 
          \bigl(
              1 
            + 
              \frac
                {1}
                {\rho + \Cl{cst_phieGoh0ain2poi6voeMi8xe} (\dist (q, \partial \domain) + \rho_* (R, \rho))}
          \bigr)
      }
    \int_{\domain} 
      \vortex (x) 
      \ln 
       \frac{1}{\rho + \Cr{cst_phieGoh0ain2poi6voeMi8xe} \dist (x, \partial \domain)}
      \dif x
  \\
  & \le
    \C
    \frac{ \rho^2\lorentznorm{\vortex (\rho \cdot)}
    +
          \frac{\norm{\Gamma}{}^2}{\abs{\Gamma}} 
          }
      {\ln 
          \bigl(
              1 
            + 
              \frac
                {1}
                {\rho + \Cr{cst_phieGoh0ain2poi6voeMi8xe} (\dist (q, \partial \domain) + \rho_* (R, \rho))}
          \bigr)
      }
      .
      \end{split}
\end{equation}
The conclusion follows from \eqref{eq_NDEzOTY2Zj}, \eqref{eq_YjQzY2RlOT} and the definition of \(\Gamma\) in \eqref{eqVortexCirculation}.
\end{proof}

\subsection{Transport of Lorentz norms}

The Lorentz norm that was defined in \eqref{eq_def_Lorentznorm} and that appeared in all the concentration estimates of this section, is invariant under transformations that preserve the measure of level sets of \(\vortex\), but by \eqref{eq_vorticity} the flow under the lake equations \eqref{eqLake} preserves
the measure with density \(b\) of level sets of \(\vortex/b\) instead.
The next proposition shows that the Lorentz norms can be kept into control. 

\begin{proposition}[Transport of Lorentz norm]
\label{proposition_transport_Lorentz}
Let \(\vortex, \tilde{\vortex} : \domain \to \reals\).
If for every \(\lambda > 0\),
\begin{equation*}
   \int_{\vortex (x) > \lambda b (x)} b (x) \dif x
 =
   \int_{\tilde{\vortex} (x) > \lambda b (x)} b (x) \dif x,
\end{equation*}
then 
  \begin{equation*} 
      \lorentznorm{\tilde{\vortex}} 
    \leq 
      \Bigl(\frac{\sup_\domain\depth}{\inf_\domain\depth} \Bigr)^2 \lorentznorm{\vortex}.
  \end{equation*}
\end{proposition}

The proof of \cref{proposition_transport_Lorentz} is based on the following geometrical computation of the Lorentz norm.

\begin{lemma}
\label{lemmaLorentz}
For every \(\vortex : \plane \to \reals\), we have
\begin{equation*}
    \lorentznorm{\vortex}         
  = 
    2 \pi
    \int_0^{+\infty}
      \int_0^{\sqrt{\abs{\{\abs{\vortex} > \lambda \}}/\pi}}
        r \Bigl(\ln \frac{1}{r}\Bigr)_+
        \dif r
      \dif \lambda .
\end{equation*}
\end{lemma}
\begin{proof}
By the Hardy--Littlewood rearrangement inequality, we have 
\begin{equation*}
      \int_{\plane}\positivepart{\ln\frac{1}{|x|}}\,|\vortex| (x)\dif x 
   \le
      \int_{\plane}\positivepart{\ln\frac{1}{|x|}}\,\abs{\vortex}^* (x)\dif x,
\end{equation*}
where \(\abs{\vortex}^* : \plane \to \reals\) is a radial function such that 
for every \(\lambda > 0\), \(\abs{\{ x \in \plane \st \abs{\vortex}^* (x) > \lambda \}} = \abs{\{ x \in \plane \st \abs{\vortex} (x) > \lambda\}}\).
We compute then 
\begin{equation*}
\begin{split}
 \int_{\plane}\positivepart{\ln\frac{1}{|x|}}\,\abs{\vortex}^* (x)\dif x
 &=
  \int_0^{+\infty}
    \int_{\abs{\vortex}^* (x) > \lambda}
      \positivepart{\ln\frac{1}{|x|}} 
      \dif x
    \dif \lambda\\
  &= 
  2 \pi
  \int_0^{+\infty}
    \int_0^{\sqrt{\abs{\{\abs{\vortex}^* > \lambda\}}/\pi}}
       r \Bigl(\ln \frac{1}{r}\Bigr)_+
  \dif r  \\
&= 
  2 \pi
  \int_0^{+\infty}
    \int_0^{\sqrt{\abs{\{\abs{\vortex} > \lambda\}}/\pi}}
       r \Bigl(\ln \frac{1}{r}\Bigr)_+
  \dif r. \qedhere
\end{split}
\end{equation*}
\end{proof}

\begin{proof}[Proof of \cref{proposition_transport_Lorentz}]
We have for every \(\lambda > 0\), by assumption
  \begin{equation*}
   \begin{split}
      \abs{\{\abs{\tilde{\vortex}} > \lambda\}}
      &\le 
        \frac{1}{m}
        \int_{\{\abs{\tilde{\vortex}}/b > \lambda/M\}} b
      = \frac{1}{m}
        \int_{\{\abs{\vortex}/b > \lambda/M\}} b
      \le 
        \alpha 
        \,
        \abs{\{\abs{\vortex} > \lambda/\alpha\}},
   \end{split}
  \end{equation*}
where \(m= \inf_{\domain} b\), \(M = \sup_{\domain} b\) and \(\alpha = M/m\).

By \cref{lemmaLorentz}, we have now
\begin{equation*}
\begin{split}
    \lorentznorm{\tilde{\vortex}}
  &=     
    2 \pi
    \int_0^{+\infty}
      \int_0^{\sqrt{\abs{\{\abs{\tilde{\vortex}} > \lambda \}}/\pi}}
        r \Bigl(\ln \frac{1}{r}\Bigr)_+
        \dif r
      \dif \lambda \\
  &\le 
    2 \pi
    \int_0^{+\infty}
      \int_0^{\sqrt{\alpha \abs{\{\abs{\vortex} > \lambda/\alpha \}}/\pi}}
        r \Bigl(\ln \frac{1}{r}\Bigr)_+
        \dif r
      \dif \lambda\\
  &= \alpha^2
  \,
    2 \pi
    \int_0^{+\infty}
      \int_0^{\sqrt{\abs{\{\abs{\vortex} > \lambda' \}}/\pi}}
        r' \Bigl(\ln \frac{1}{\sqrt{\alpha} r'}\Bigr)_+
        \dif r'
      \dif \lambda'  \\
  & \le 
  \alpha^2
  \,
    2 \pi
    \int_0^{+\infty}
      \int_0^{\sqrt{\abs{\{\abs{\vortex} > \lambda' \}}/\pi}}
        r' \Bigl(\ln \frac{1}{r'}\Bigr)_+
        \dif r'
      \dif \lambda'.  \qedhere
\end{split}
\end{equation*}
\end{proof}

Finally we estimate the behaviour of the Lorentz norm under rescaling on the domain.

\begin{proposition}
\label{propositionLorentzScaling}
If \(\sigma \in (0, + \infty)\) and \(\vortex : \domain \to \reals\), then 
\begin{equation*}
    \lorentznorm{\vortex(\sigma \cdot)}
  \le 
    \frac{1}{\sigma^2} 
    \Bigl(
        \lorentznorm{\vortex}
      + 
        (\ln \tfrac{1}{\sigma})_+ \norm{\vortex}{L^1}
    \Bigr).
\end{equation*}
\end{proposition}

\begin{proof}
We have, by \cref{lemmaLorentz},
\begin{equation*}
    \lorentznorm{\vortex (\sigma \cdot)}         
  = 
    2 \pi
    \int_0^{+\infty}
      \int_0^{\sqrt{\abs{\{\abs{\vortex} > \lambda \}}/(\sigma^2 \pi)}}
        r \Bigl(\ln \frac{1}{r}\Bigr)_+
        \dif r
      \dif \lambda
\end{equation*}
and by a change of variable:
\begin{equation*}
    \lorentznorm{\vortex (\sigma \cdot)}         
  = \frac{2 \pi}{\sigma^2}
    \int_0^{+\infty}
      \int_0^{\sqrt{\abs{\{\abs{\vortex} > \lambda \}}/\pi}}
        r' \Bigl(\ln \frac{1}{\sigma r'}\Bigr)_+
        \dif r'
      \dif \lambda,
\end{equation*}
From this, we conclude
\begin{equation*}
\begin{split}
    \lorentznorm{\vortex (\sigma \cdot)}     
  &\le \frac{2 \pi}{\sigma^2}
  \Bigl(
    \int_0^{+\infty}
      \int_0^{\sqrt{\abs{\{\abs{\vortex} > \lambda \}}/\pi}}
        r' \Bigl(\ln \frac{1}{r'}\Bigr)_+
        \dif r'
      \dif \lambda \\
      &\hspace{2cm} +
    \int_0^{+\infty}
      \int_0^{\sqrt{\abs{\{\abs{\vortex} > \lambda \}}/\pi}}
        r' \Bigl(\ln \frac{1}{\sigma}\Bigr)_+
        \dif r'
      \dif \lambda 
   \Bigr)
\end{split}
\end{equation*}
and the conclusion follows.
\end{proof}

\section{Asympotics evolution of vortices}
\label{section_Asymptotic}

\subsection{Asymptotic representation of derivatives}

In order to study the evolution of vortices, we will need to differentiate several quantities of the form 
\[
  \int_{\domain} \vortex (t)\, \eta,
\]
where \(\eta \in C^1 (\Bar{\domain})\) is a given spatial test function.

\begin{proposition}
\label{propTotalVorticityDerivative}
If \((\vortex, \velocity) \in L^\infty(\reals \times \domain) \times L^\infty (\reals, L^2 (\domain))\) is a weak solution to the vorticity formulation lake equation and if \(\eta \in C^1 (\bar{\domain})\), then the function \(t \in \reals \mapsto \int_{\domain} \vortex (t)\, \eta\) is weakly differentiable and for almost every \(t \in \reals\),
\begin{equation*}
  \frac{d}{dt} \int_{\domain} \vortex (t)\, \eta 
  =
    \int_{\domain} 
      \vortex (t) 
      \, 
      \velocity (t) 
      \cdot 
      \nabla \eta
  = 
    - 
      \int_{\domain}
        \vortex (t)\, \frac{\nabla( \mathcal{K}_b [\vortex (t)])}{b} \times \nabla \eta
    - 
      \sum_{i = 1}^\nbislands 
        \Gamma_i
        \int_{\domain}
          \vortex (t) \, \nabla \psi_i \times \nabla \eta.
\end{equation*}
\end{proposition}
\begin{proof}
Given \(\theta \in C^1_c ([0, +\infty))\), 
we apply \cref{proposition_test_functions} to the test function \(\theta \eta\) 
and we obtain
\[
  \theta (0)
  \int_{\domain}
      \vortex_0\,
      \eta
      +
  \int_0^{+\infty} 
   \int_{\domain}
      \vortex (t) \bigl(
      \theta' (t) \eta
  +
      \theta (t)\, \velocity (t)  \cdot \nabla \eta \bigr) \dif t
    = 0.
\]
The second identity follows then from \eqref{eq_velocity_reconstruction}.
\end{proof}

Under the additional assumption that the function \(\eta\) is constant on each component of the boundary \(\partial \domain\), 
we obtain a representation in which the gradient of the stream function is replaced by the gradient of the depth function.

\begin{proposition}
\label{propositionIntegrationByParts}
If \(\eta \in C^2 (\Bar{\domain})\) is constant on each component of \(\partial \domain\), 
then there exists a constant \(C > 0\) such that for every \(\vortex \in L^\infty (\domain)\),
\begin{equation*}
    \strabs{
        \int_{\domain}
          \vortex \frac{\nabla( \mathcal{K}_b [\vortex])}{b} \times \nabla \eta  
      - 
        \frac{1}{2}
        \int_{\domain}
            \vortex
            \,
            \mathcal{K}_b [\vortex] 
            \frac{\nabla b}{b^2} 
          \times 
            \nabla \eta
    }      
  \le     
    C
    \abs{\Gamma}^2
.
\end{equation*}
\end{proposition}

\begin{proof}
[Proof of \cref{propositionIntegrationByParts}]
By \cref{velocityIntegralExpansion},
we have the following identity:
\begin{multline}
\label{eq_Three_Term_Dcomposition}
 \int_{\domain}
      \vortex \frac{\nabla( \mathcal{K}_b [\vortex])}{b} \times \nabla \eta 
      - \frac{1}{2} \int_{\domain}
            \vortex
            \,
            \mathcal{K}_b [\vortex] 
            \frac{\nabla b}{b^2} 
          \times 
            \nabla \eta\\
  =
     \iint\limits_{\domain\times\domain}
        \gradient\greenlaplace(x,y)
        \times 
        \nabla \eta (x) 
        \,
        \vortex(x)
        \,
        \vortex(y)
        \sqrt{\frac{\depth(y)}{\depth(x)}}
        \dif x \dif y\\
    + 
      \iint
        \limits_{\domain\times\domain}
          \frac{\gradient R_b (x,y)}{b}
        \times 
          \gradient\eta (x)
        \,
        \vortex(x)
        \,
        \vortex(y)
        \dif x 
        \dif y.
\end{multline}
For the second term on the right-hand side of \eqref{eq_Three_Term_Dcomposition}, 
we have by the boundedness properties of the derivatives of the function \(R_b\),
\begin{equation}
\label{eq_total_vorticity_3d}
    \strabs{
      \;
       \iint
        \limits_{\domain\times\domain}
          \frac{\gradient R_b (x,y)}{b}
        \times 
          \gradient\eta (x)
          \,
        \vortex(t,x)
        \,
        \vortex(t,y)
        \dif x 
        \dif y
    \,
    }
  \le 
    \C
    \abs{\Gamma}^2.
\end{equation}

We now estimate the first term on the right-hand side of \eqref{eq_Three_Term_Dcomposition}.
By symmetry, we have 
\begin{multline}
\label{eqf095b444}
    \iint
      \limits_{\domain\times\domain}
      \sqrt{\frac{\depth(y)}{\depth(x)}}
      \big(\crossproduct{\gradient\greenlaplace(x,y)}{\gradient\eta (x)}\big)\,
      \vortex(x)
      \,
      \vortex(y)
      \dif x \dif y\\
  = \frac{1}{2}
    \iint\limits_{\domain\times\domain}
      \sqrt{b (x) b (y)}
      \Big(
          \crossproduct{\gradient\greenlaplace(x,y)}{\frac{\gradient\eta(x)}{\depth (x)}} 
       + \crossproduct{\gradient\greenlaplace(y,x)}{\frac{\gradient\eta(y)}{\depth (y)}}
      \Big) 
      \,
      \vortex(x)
      \,
      \vortex(y)
      \dif x 
      \dif y.
\end{multline}
For every \(x, y \in \domain\), we have 
\begin{multline}
\label{eqf9c9a785}
    {\gradient\greenlaplace(x,y)} 
  \times 
    {\frac{\gradient\eta(x)}{\depth (x)}} 
  + 
    {\gradient\greenlaplace(y,x)}
  \times {\frac{\gradient\eta(y)}{\depth (y)}}
  \\
  =
    \gradient\greenlaplace(x,y)
    \times 
    \biggl(\frac{\nabla \eta (x)}{b (x)}
    - \frac{\nabla \eta (x)}{b (x)}
    \biggr)\\
   +
    \left(\gradient\greenlaplace(x,y) + \gradient\greenlaplace(y,x)\right)
    \times \frac{\nabla \eta (y)}{b (y)}
    \biggr).  
\end{multline}
By \cref{proposition_Gradient_bound}, since by assumption \(\nabla \eta\) and \(b\) are both Lipschitz-continuous, we have 
\begin{equation}
\label{eq47dd72ad}
 \biggabs{\gradient\greenlaplace(x,y)
    \times 
      \Bigl(
          \frac{\nabla \eta (y)}{b (y)} 
        - 
          \frac{\nabla \eta (x)}{b (x)}
      \Bigr)  
}
\le \C.
\end{equation}
For the other contribution \cref{proposition_symmetric_gradient_estimate_boundary},
we have if \(\abs{y - x} + \dist(x, \partial \domain) + \dist (y, \partial \domain)\le \delta\),
\begin{equation*}
    \Bigg|
          \gradient\greenlaplace[\domain](x,y)
        +
          \gradient\greenlaplace[\domain](y,x)
   - 2 \frac
            {
                x - P_{\partial \domain} (x) 
              }
            {
              \pi 
              (
                  \abs{x - y}^2 
                + 
                  4 
                  \dist (x, \partial \domain) 
                  \dist (y, \partial \domain)
              )
            }
     \Bigg|
  \le
    \C 
\end{equation*}
and 
\[
 (x - P_{\partial \domain} (x) )
            \times 
            \nabla \eta (P_{\partial \domain} x) = 0,
\]
and thus
\begin{equation}
\label{eq_eef5Yoo5nahfooliechaecho}
\strabs{
 (\gradient\greenlaplace(x,y) + \gradient\greenlaplace(y,x))
    \times \frac{\nabla \eta (y)}{b (y)}}
\le 
\C \biggl( 1 + \frac{\dist(x, \partial \domain)^2}{\abs{x - y}^2 + \dist (x, \partial \domain) \dist (y, \partial \domain)}\biggr)
\le \C,
\end{equation}
if \(\dist (x, \partial \domain) \le \dist (y, \partial \domain)\).
The case where \(\dist (y, \partial \domain) \le \dist (x, \partial \domain)\) follows symmetrically.

If  \(\abs{y - x} + \dist(x, \partial \domain) + \dist (y, \partial \domain)\ge \delta\), by \cref{proposition_symmetric_gradient_estimate}, we have
\begin{equation}
\label{eq42f66c92}
\Bigabs{
 \bigl(\gradient\greenlaplace(x,y) + \gradient\greenlaplace(y,x)\bigr)
    \times \frac{\nabla \eta (x)}{b (x)}
    }
    \le \C.
\end{equation}
By combining \eqref{eqf095b444}, \eqref{eqf9c9a785},  \eqref{eq_eef5Yoo5nahfooliechaecho}, \eqref{eq47dd72ad} and \eqref{eq42f66c92}, we deduce that 
\begin{equation}
\label{eq_total_vorticity_2d}
  \strabs{
    \;
    \iint\limits_{\domain\times\domain}
      \sqrt{\frac{\depth(y)}{\depth(x)}}
      \,
      \bigl(
        \crossproduct
          {\gradient\greenlaplace(x,y)}
          {\gradient\eta(x)}
      \bigr)
      \,
      \vortex(x)\,\vortex(y)\dif x \dif y
      \;
      }
  \le 
    \C 
    \,
    \abs{\Gamma}^2.
\end{equation}
The conclusion follows from the combination of the identity \eqref{eq_Three_Term_Dcomposition} with the inequalities \eqref{eq_total_vorticity_3d} and \eqref{eq_total_vorticity_2d}.
\end{proof}

\subsection{Asymptotic conservation of the total vorticity}

For the lake equations \eqref{eqLake}, the \emph{total vorticity}, defined by \eqref{eq_def_totalvorticity}
is not conserved in general. Indeed, by \cref{propTotalVorticityDerivative}, one has for almost every \(t \in \reals\) 
  \begin{equation*} 
      \totalvorticity' (t) 
    = 
      \int_{\domain}
        \vortex (t)
        \,
        \scalarproduct{\velocity(t)}{\gradient\depth}{\plane},
  \end{equation*}
and there is no reason for the right-hand side to vanish.
On the other hand Richardson's formal law \eqref{eq_Richardson} suggests that vorticity should follow level lines of the depth and thus one can hope the total vorticity to be asymptotically preserved. 
The following result gives a bound on the variation of the total vorticity during the motion.

\begin{proposition}%
[Asymptotic conservation of the total vorticity]%
\label{conservationDepth}
If the function \(b \in C^2 (\bar{\domain})\) is constant on each connected component of the boundary, 
then there exists a constant \(C>0\) that depends only on \(\domain\) and \(\depth\), such that if \((\vortex, \velocity) \in L^\infty(\reals \times \domain) \times L^\infty (\reals, L^2 (\domain))\) is a weak solution to the vorticity formulation lake equation then we have for almost every \(t \in \reals\):
  \begin{equation*} 
    \bigabs{\totalvorticity(t) - \totalvorticity(0)}
    \leq C \abs{\Gamma}\, \norm{\Gamma}{}\,\abs{t} .
  \end{equation*}
\end{proposition}

The conclusion of \cref{conservationDepth} can be rewritten as 
\begin{equation*}
    \frac
      {\bigabs{\totalvorticity(t) - \totalvorticity(0)}}
      {\abs{\vortexstrength}}
    \leq 
      C
      \,
      \frac{\abs{\vortexstrength}\norm{\vortexstrength}{}}{E}\,
      \frac{E \abs{t}}{\abs{\Gamma}}
      ;
\end{equation*}
in the regime where \(\frac{\energy}{\abs{\vortexstrength} \norm{\Gamma}{}}\to+\infty\),
 the above estimate may be interpreted as stating that at the time scale \(\Gamma/E\) the variations of the total vorticity are much smaller than the total circulation.

\begin{proof}[Proof of \cref{conservationDepth}]
By \cref{propTotalVorticityDerivative} with \(\eta = b\), the function \(\totalvorticity\) is weakly differentiable and 
for almost every \(t \in \reals\),
\begin{equation}
  \label{eq_OGQ5ZDNiNj}
  \totalvorticity' (t) 
    =
      -
      \int_{\domain}
        \frac{\nabla (\mathcal{K}_b [\vortex (t)]) \times \gradient\depth}{b}
        \,
        \vortex(t)
      - 
      \sum_{i = 1}^\nbislands
          \Gamma_i
          \int_{\domain} 
            \frac{\nabla \psi_i \times \nabla b}{b} \,\vortex (t)
.
\end{equation}
Since we have assumed the bathymetry function \(b\) to be constant on the boundary, we apply \cref{propositionIntegrationByParts} with \(\eta = b\) and we obtain, since 
\(\crossproduct{\gradient\depth}{\gradient\depth}=0\), 
\begin{equation}
\label{eq_ODgwMDg5NT}
    \strabs{\,
      \int_{\domain}
        \frac{\nabla (\mathcal{K}_b [\vortex(t)]) \times \gradient\depth}{b}
        \,
        \vortex(t)
    }
  \le
    \C
    \abs{\Gamma}^2
    .
\end{equation}
By the boundedness properties of the gradients of \(b\) and \(\psi_i\) and by definition of \(\norm{\Gamma}{}\) in \eqref{eq_norm_Gamma}, we also have
\begin{equation}
\label{eq_total_vorticity_4th}
\strabs{\sum_{i = 1}^\nbislands
          \Gamma_i
          \int_{\domain} 
            \frac{\nabla \psi_i \times \nabla b}{b} \,\vortex (t)
      }
  \le 
    \C
    \,
    \abs{\Gamma}
    \,
    \norm{\Gamma}{}.
\end{equation}

By \eqref{eq_OGQ5ZDNiNj}, \eqref{eq_ODgwMDg5NT} and \eqref{eq_total_vorticity_4th} we deduce in view of \cref{propTotalVorticityDerivative}
that for almost every \(t \in \reals\),
\begin{equation*}
  \bigabs{\totalvorticity' (t)}
  \le 
    \C \,
    \abs{\Gamma}
    \,
    \norm{\Gamma}{}
    .
\end{equation*}
The conclusion then follows by integration.
\end{proof}

\subsection{Evolution of a singular vortex}

We are now in position to state and prove the main result of the present work.

\begin{theorem}
[Evolution of the vortex core]
\label{EvolutionLaw}
Let \(\domain\subseteq\plane\) be a bounded domain of class \(C^2\) and \(b \in C^2(\closure{\domain}, (0, +\infty))\).
Assume that \(b\) is constant on each component of \(\partial D\).
If 
\begin{enumerate}[(a)]
  \item 
    \((\velocity^n)_{n > 0}\) is family of weak solutions to the lake equations \eqref{eqLake},
  \item 
    \(\vortex^n (0) \ge 0\) almost everywhere on \(\domain\),
  \item 
  \label{assumption_c}
    there exists \(q_0 \in D\) such that for every \(\eta \in C (\domain)\),
    \begin{equation*}
        \lim_{n \to \infty}
          \frac
            {1}
            {\vortexstrength^n}
          \int_{\domain} 
            \eta 
            \,
            \vortex^n (0) 
        =
          \eta (q_0),
    \end{equation*}
    %
  \item \label{assumption_f}
    \(\displaystyle \sup_{n \in \integers} \frac{1}{\abs{\vortexstrength^n}} \lorentznorm{(\rho^n)^2 \vortex^n (0, \rho^n\cdot)} < +\infty\) where \(\rho^n =       
      \exp
        \bigl( 
          -\frac
            {4\pi\,\energy^n}
            {\vortexstrength^n\ \totalvorticity^n} 
        \bigr)
    \),
    \item
      \(\displaystyle \sup_{n \in \integers} \frac{\norm{\vortexstrength^n}{}}{\abs{\vortexstrength^n}} < + \infty\),
\end{enumerate}
and let \(q_* : \reals \to \domain\) be the unique solution to 
the Cauchy problem
  \begin{equation*} 
    \left\{
      \begin{aligned}
          \dot{q}_*(s) 
        &= 
          -\flipgradient\Bigl(\frac{1}{\depth}\Bigr) \bigl(q_* (s)\bigr)
        & & \text{if \(s \in \reals\)},
        \\
          q_* (0)
        &=
          q_0 ,
      \end{aligned} 
    \right.
  \end{equation*}
then one has, for every \(\varphi \in C^1 (\Bar{\domain})\), uniformly in \(s \in \reals\) over compact subsets,
    \begin{equation*}
        \lim_{n \to \infty}
          \frac
            {1}
            {\vortexstrength^n}
          \int_{\domain} 
            \varphi 
            \,
            \vortex^n (\vortexstrength^n s/\energy^n) 
        =
          \varphi (q_* (s))
    \end{equation*}
    and 
    \begin{equation*}
      \lim_{n \to \infty}
      \frac
      {1}
      {\energy^n}
      \int_{\domain} 
      \varphi 
      \,
      \abs{\velocity^n}^2 (\vortexstrength^n s/\energy^n) 
      =
      \varphi (q_* (s)).
    \end{equation*} 
\end{theorem}  

\Cref{EvolutionLaw} implies immediately \cref{truncatedEvolutionLaw}. Indeed, it suffices to observe that by \eqref{eq_def_Lorentznorm} the Lorentz norm is controlled by the \(L^\infty\) norm and that \eqref{assumption_f} is trivially satisfied since when \(D\) is simply connected, \(\nbislands = 0\) and \(\norm{\Gamma}{} = \abs{\Gamma}\) by definition in \eqref{eq_norm_Gamma}.

\begin{proof}
[Proof of \cref{EvolutionLaw}]
We first note that by \cref{velocityIntegralExpansion}, and \eqref{eq_definition_regular_part} and by \cref{proposition_Energy_Formula}, we have for every \(\delta > 0\) such that \(\overline{B (q_0, \delta)} 
\subset \domain\),
\begin{multline*}
\frac{1}{4\pi}
\ln \frac{1}{2\delta}
\Biggl(\int_{\overline{B (q_0, \delta)}} \sqrt{b} \, \vortex^n (0)\Biggr)^2 
\le 
\int_{B (q_0, \delta)}
\frac{1}{4 \pi} \ln 
\frac{1}{\abs{y - x}}
    \sqrt{b (x) b (y)}\, \vortex^n (0, x) \,\vortex^n (0, y)\dif y \dif x\\
\\
\le 
\frac{1}{2}
\int_{\domain} \vortex \,\mathcal{K}_b [\vortex] 
+ \C \norm{\vortexstrength^n}{}^2 \le \energy^n (0) + \C \norm{\vortexstrength^n}{}^2.
\end{multline*}
so that by the assumptions \eqref{assumption_c} and \eqref{assumption_f}, for every \(\delta > 0\) small enough we have 
\[
\liminf_{n \to \infty} \frac{\energy^n}{(\vortexstrength^n)^2} \ge \frac{1}{4\pi}
\ln \frac{1}{2\delta}
\]
and thus 
\begin{equation}
    \label{assumption_d}
    \lim_{n \to \infty} \frac{\energy^n}{(\vortexstrength^n)^2} = + \infty.
  \end{equation}

By definition of \(\rho^n (t)\) in \eqref{defTypscale}, we observe that for each \(n \in \integers\) and \(t \in \reals\),
\begin{equation*}
    \rho^n (t) 
  = 
    \exp 
      \Bigl( 
        - \frac
          {4\pi\,\energy^n}
          {\vortexstrength^n\ \totalvorticity^n (t)}
      \Bigr)
   \le 
      \exp 
        \Bigl( -\frac{4\pi\,\energy^n}{\abs{\vortexstrength^n}^2 \sup_{\domain} b}\Bigr),
\end{equation*}
and thus by \eqref{assumption_d}, we have \(\rho^n \to 0\) uniformly on \(\reals\) as \(n \to \infty\).
Moreover, we have by definition of \(\rho^n (t)\) and by \cref{conservationDepth} for every \(n \in \integers\) and \(t \in \reals\),
\begin{equation}
\label{eq_aigh2eiyohtaumashushieSe}
    \strabs{
      \ln 
        \frac
          {\rho^n (t)}
          {\rho^n (0)}
    }
  =
    \frac
      {
        4\pi\,
        \energy^n 
        \,
        \abs{\totalvorticity^n (t) - \totalvorticity^n (0)}
      }
      {
        \vortexstrength^n
        \, 
        \totalvorticity^n (0)
        \,
        \totalvorticity^n (t)
      }
  \le
    \frac       
      {
        \Cl{cst_id8x90R}
        \,
        \energy^n \abs{t}
      }
      {\vortexstrength^n} . 
\end{equation}
By \cref{propositionLorentzScaling}, by \eqref{eq_aigh2eiyohtaumashushieSe},
and then \cref{proposition_transport_Lorentz}, this implies that for every \(n \in \integers\) and \(t \in \reals\)
\begin{equation}
\label{eq_Cobu9eipez}
\begin{split}
    \rho^n (t)^2
    \,
    \lorentznorm{\vortex^n (t, \rho^n (t) \cdot)}
& \le \rho^n (0)^2
        \lorentznorm{\vortex^n (t, \rho^n (0) \cdot)}
      + 
        \left(\ln\frac{\rho^n (t)}{\rho^n (0)}\right)_+ \norm{\omega_n}{L^1}\\
 & \le 
  \rho^n (0)^2
        \lorentznorm{\vortex^n (t, \rho^n (0) \cdot)}
      + 
        \Cr{cst_id8x90R} 
        \energy^n \abs{t}
 \\
  &\le 
    \C  
    \Bigl(
        \rho^n (0)^2 \lorentznorm{\vortex^n (0, \rho^n (0) \cdot)}
      + 
        \energy^n \abs{t}\bigr)
    \Bigr)\\
  &\le 
    \C
    \,
    \abs{\vortexstrength^n}
\Bigl(1 + \frac{\energy^n \abs{t}}{\abs{\vortexstrength^n}}\Bigr),
\end{split}
\end{equation}
in view of our assumption \eqref{assumption_f}. Since $\rho^n(t)$ uniformly converges
to $0$ as $n\to\infty$, one can choose $R^n(t)\defeq 1/\sqrt{\rho^n(t)}$ for sufficiently
large $n\in\integers$ in \cref{propositionConfinementCenter}. For every \(S > 0\), there exists then \(\delta > 0\) 
such that if \(n \in \integers\) is large enough and if \(\energy^n \abs{t}/\vortexstrength^n \le S\), then \(q^n (t) \in \domain\) and \(\dist (q^n (t), \partial \domain) \ge \delta \).

Let \(\eta\in C^\infty(\domain)\) be a positive function bounded by \(1\),
such that \(\eta(x) = 0\) if \(\distance(x,\boundary\domain)\leq\delta/3\) and \(\eta(x) = 1\) if \(\distance(x,\boundary\domain)\geq 2\delta/3\).
We define the truncated center of vorticity \(\Tilde{q}^n : \reals \to \reals^2\) by setting for each \(t \in \reals\) and \(n \in \integers\),
  \begin{equation*} 
      \Tilde{q}^n(t) 
    \defeq 
      \frac{1}{\vortexstrength^n}
      \int_{\domain}
        \eta(x)
        \,
        x
        \, 
        \vortex^n (t,x)
        \dif x .
  \end{equation*}
We observe that for every \(n\) large enough, \(\rho_*^n(t)(R^n (t), \rho^n (t)) \le \frac{\delta}{3}\) and
\begin{equation*}
    \int_{\dist (x, \partial \domain)  \le \delta} \vortex^n
  \le 
    \int_{\domain\setminus B\big(q (t), \rho_*^n (t) (R^n (t), \rho^n (t)) \big)}
      \vortex^n (t, x) 
      \dif x  .
\end{equation*}
By \cref{proposition_cvort_conc}, we obtain
  \begin{equation}
    \label{eqefiAn3XRwSir}
    \begin{split}
        \big| q^n (t)- \Tilde{q}^n(t)\big| 
      &
      \leq 
        \frac{1}{\vortexstrength^n}
          \int_{\domain\setminus B\big(q (t), \rho_*^n (t) (R^n (t), \rho^n (t)) \big)}
          \vortex^n(t,x)\dif x 
      \\
      &
      \le 
        \frac{\C}{\ln \bigl(R^n(t)\bigr)}
        \biggl( \rho^n (t)^2 
            \frac{\lorentznorm{\vortex^n(t, \rho^n (t) \cdot)}}{\vortexstrength^n}
          + 
            1
        \biggr)
.
    \end{split}
  \end{equation}
By \eqref{eq_Cobu9eipez} and by the choice of $R^n(t)$ it follows that 
\((\Tilde{q}^n (\vortexstrength^n \cdot/\energy^n) - q^n (\vortexstrength^n \cdot/\energy^n))_{n \in \integers}\)
converges uniformly to \(0\) over \([-S, S]\).

By \cref{propTotalVorticityDerivative}, we have for almost every \(t \in \reals\),
  \begin{equation*} 
        \dot{\Tilde{q}}^n (t)
      =
         \frac{1}{\vortexstrength^n}\Bigg(
          \int_\domain
            \frac{\flipgradient \mathcal{K}_b[\vortex^n (t)] \cdot \nabla \xi}{b}
            \ 
            \vortex^n (t) 
        + 
          \sum_{i = 1}^\nbislands
            \Gamma_i^n
            \int_{\domain} \frac{\flipgradient \psi \cdot \nabla \xi}{b} 
            \
            \vortex^n (t)\Bigg),
  \end{equation*}
where the vector field \(\xi \in C^\infty_c (\domain, \reals^2)\) is defined for each \(x \in \domain\) by \(\xi (x) \defeq x\, \eta (x)\).
In view of \cref{propositionIntegrationByParts},
we have 
\begin{equation}
\label{eq_4IApVpHBMC3}
\biggabs{
 \dot{\Tilde{q}}^n (t)
    - 
      \frac{1}{2\vortexstrength^n}
      \int_{\domain} 
        \frac{\flipgradient\depth \cdot \gradient \xi}{\depth^2}\,
      \vortex^n (t) \,
      \mathcal{K}_b [\vortex^n (t)]\,
}
  \le 
    \C 
    \norm{\vortexstrength^n}{}.
\end{equation}

Now we observe that since \(\nabla b\) is Lipschitz-continuous and since \(\xi (x) = x\) if \(\distance (x, \partial \domain) \ge \delta\), we have for every \(x \in \domain\) and \(y\) such that \(\dist (y, \partial \domain) \ge \delta\),
\[
    \biggabs{
        \frac{\flipgradient\depth (x) \cdot \nabla \xi (x)}{\depth(x)^2}
      - 
        \frac{\flipgradient b (y)}{b (y)^2} 
    }
  \le 
    \Cl{cst_Oovei9Josh}
    \abs{x - y}
\]
and thus 
\begin{multline}
\label{eq_IJ8Ntv0dug7V}
    \biggabs{
      \int_{\domain} 
        \Bigl(
            \frac{\flipgradient\depth \cdot \nabla \xi}{\depth^2}
          - 
            \frac{\flipgradient b (q^n (t))}{b^2 (q^n (t))} 
        \Bigr)
        \vortex^n (t) 
        \,
        \mathcal{K}_b [\vortex^n(t)]
        \,
      }\\
  \le 
  \Cr{cst_Oovei9Josh}
  \biggl(
    \int_{\domain} 
      \abs{x - q^n (t)}
      \,
      \abs{\vortex^n (t,x)} 
      \dif x\biggr)
      \,
    \norm
      {\mathcal{K}_b [\vortex^n](t)}
      {L^\infty (\domain)}.
\end{multline}
We have by a direct bound
\begin{equation}
\label{eq_HRVTP4DLWHE}
    \int_{\domain \cap B (q^n (t), \rho_*^n (t) (R^n (t), \rho^n (t)))} 
      \abs{x - q^n (t)}
      \,
      \abs{\vortex^n (t,x)}
      \dif x
  \le 
    \C
    \,
    \rho_*^n (t) (R^n (t), \rho^n (t))
    \,
    \abs{\vortexstrength^n}
\end{equation}
and by \cref{proposition_cvort_conc},
\begin{multline}
\label{eq_GfKt6hBrkpuu}
      \int_{
            \domain 
          \setminus 
            B (q^n (t), \rho_*^n (t) (R^n (t), \rho^n (t)))} 
        \abs{x - q^n (t)}
        \,
        \abs{\vortex^n(t,x)} 
        \dif x
  \\\le 
    \frac
      {\C}
      {\ln R^n(t)}
    \bigl( 
        \rho^n (t)^2 
          \lorentznorm{\vortex(\rho^n (t) \cdot)} 
        + 
          \abs{\vortexstrength^n} 
    \bigr).
\end{multline}
Thus we have by \eqref{eq_IJ8Ntv0dug7V}, \eqref{eq_HRVTP4DLWHE}, \eqref{eq_GfKt6hBrkpuu} and by \cref{controlCondition},
\begin{multline}
\label{eq_weLtgetDzpZI}
    \strabs{
      \int_{\domain} 
        \Bigl(
            \frac{\flipgradient\depth \cdot \nabla \xi}{\depth^2}
          - 
            \frac{\flipgradient b (q^n (t))}{b^2 (q^n (t))} 
        \Bigr)
        \vortex^n (t) 
        \,
        \mathcal{K}_b [\vortex^n(t)]
        \,}\\
  \le 
    \C
    \,
    \biggl( 
        \bigl(
            \rho_*^n (t) (R^n(t), \rho^n (t) )
        \bigr)
        \,
        \abs{\vortexstrength^n}
      + 
        \frac{1}{\ln R^n(t)}
        \,
        \Bigl( 
            \rho^n (t)^2 
            \lorentznorm{\vortex^n (t, \rho^n (t) \cdot)} 
          + 
            \abs{\vortexstrength^n}
        \Bigr)
      \biggr).
\end{multline}
Finally by \cref{proposition_Energy_Formula}, we have 
\begin{equation}
\label{eq_Eetho1ooth9shu7meegainei}
\strabs{
        \energy^n - 
        \vortex^n (t) 
        \,
        \int_{\domain} \mathcal{K}_b [\vortex^n(t)]
        \,}
        \le \C \norm{\Gamma^n}{}^2.
\end{equation}
Summarizing \eqref{eq_4IApVpHBMC3}, \eqref{eq_weLtgetDzpZI} and \eqref{eq_Eetho1ooth9shu7meegainei}, we conclude that 
\begin{equation*}
    \strabs{
        \frac{\vortexstrength^n}{\energy^n} \dot {\Tilde{q}}^n (\vortexstrength^n s/\energy^n) 
      - 
        \frac{\nabla^\perp b (\Tilde{q}^n (\vortexstrength^n s/\energy^n))}{b^n (\Tilde{q}^n (\vortexstrength^n s/\energy^n))^2}
    } 
  \to 
    0,
\end{equation*}
uniformly over \(s \in [-S, S]\).

This implies in turn that \(\Tilde{q}^n (\vortexstrength^n \cdot/\energy^n)$
converges uniformly on compact subsets of $\reals$ to \(q_*\).
Finally, we conclude that \(q^n (\vortexstrength^n \cdot/\energy^n) \to q_*\) uniformly over compact subsets of \(\reals\).
By \eqref{eq_GfKt6hBrkpuu}  the narrow convergence of vorticity measures follows.

The convergence of energy density measures then follows from \cref{proposition_energy_concentration}.
\end{proof}

\section{Open problems}

\label{section_Problems}

The present work has given a first description of the asymptotic vortex dynamics for the lake equations \eqref{eqLake}.
The setting in which we have been working does not cover the whole spectrum of physically relevant situations and suggests for future research some problems that we could not tackle with the techniques that we have developped here.

A first problem would be to determine whether \cref{EvolutionLaw} holds when shore of the lake is a beach rather than a cliff, that is when \(b\) goes smoothly to \(0\) on the boundary.

\begin{openproblem}
Does a single vortex follow asymptotically the level lines of the depth \(b\) when \(\inf_{\domain} b = 0\)? 
\end{openproblem}

Whereas in our proof the assumption that \(\inf_\domain b > 0\) plays a role in the construction of the Green function and in keeping control on the Lorentz norm, stationary results cover the case of where the depth \(b\)
behaves like a power of the distance function close to the boundary \citelist{\cite{deValeriola_VanSchaftingen_2013}*{\S 3}\cite{Dekeyser1}\cite{Dekeyser2}}
(seel also chapters~5,~6).

\begin{openproblem}
  \label{problem_oox3aishoteejooZ7}
Does a single vortex follow asymptotically the level lines of the depth \(b\) when the domain is unbounded?  
\end{openproblem}

The boundedness of the domain and of \(b\) is used mainly in the construction and estimates on the Green functions.

The probably most accessible case would be when \(\domain = \reals^2\) and \(b\) is constant outside a compact set; an interesting result would cover the case where \(\domain = [0, + \infty) \times \reals\), with \(b (r, z) = r\), corresponding to the construction of vortex rings for the three-dimensional Euler equations (see \cite{Benedetto_Gaglioti_Marchioro_2000}). 

The solution of \cref{problem_oox3aishoteejooZ7} would also show that the evolution of clifford tori in the binormal curvature flow \cite{Khesin_Yang}.

Another problem would be the case of non-constant boundary values of \(b\).
\begin{openproblem}
Does a single vortex follow asymptotically the level lines of the depth \(b\) when \(b\) is not constant on the boundary? 
\end{openproblem}

Currently, the constancy plays a crucial role in the proof and the application of \cref{propositionIntegrationByParts}.

An issue with this setting is that the limiting equation would suggest vortices exiting the domain in finite time. This would not be consistent with the conservation of circulation. A possible solution to this paradox is that the interaction with the boundary at very short range perturbs strongly the asymptotics and makes the law of movement invalid.

When the lake has a flat bottom, that is when \(b\) is constant on some region, our results do not give an interesting description of the movement of the vortices, that occurs on a larger time scale. In analogy with the planar Euler equation, which corresponds to the case where \(b\) is constant on the whole domain, we expect this movement to occur at a time-scale of the order \(1/\Gamma\).

\begin{openproblem}
\label{problem_flat_bottom}
Describe the movement of a single vortex in a flat region of the lake at time scales of the order \(1/\Gamma\).
\end{openproblem}

We expect this to be described by some sort of Green function adapted to the problem. A similar second-order asymptotic description was already given for stationnary vortex pairs~\cite{Dekeyser2}
(see chapter~6).
One question is whether the movement depends only on the shape of the set on which \(b\) is flat or whether it depends fully on \(b\) and on \(\domain\).
The first scenario would be consistent with results for an analogous Ginzburg--Landau problem with discontinuous pinning \cite{DosSantos_Misiats_2011}.

Finally, it would be natural to consider the problem where the vorticity concentrates in several regions.

\begin{openproblem}
Do solutions whose initial vorticity concentrates at several points have these vortex patches following level lines of \(b\)?
\end{openproblem}

This situation is not accessible to our proofs because we characterize the size of the vortex region by some global integral quantities.

An issue raised by this problem would be possible collision of vortices moving on the same line. 
They would probably interact at a small scale and produce potentially a vortex pair whose movement might be governed by a different equation and might have a different characteristic timescale. A similar related problem would be the description of vortex pairs.

\appendix

\section{Weak solutions of the transport equation}

A first interesting fact is that for a transport equation with no flux through the boundary, it is equivalent to test the equation against compactly supported smooth functions or functions that are smooth up to the boundary.

\begin{proposition}
\label{lemma_Transport_TestFunctions}
Assume that \(\velocity \in L^\infty (W^{1, 1} (\domain))\) and that \(\velocity \cdot \normal = 0\) on \(\partial \domain\) in the sense of traces.
If \(f_0 \in L^\infty (\domain)\) and \(f \in L^\infty ([0, +\infty) \times \domain)\)
satisfy for every \(\varphi \in C^1_c ([0, +\infty)\times \domain)\) the identity 
\[
   \int_0^{+\infty}\int_{\domain} f\,(\partial_t \varphi + \velocity \cdot \nabla \varphi) 
    + \int_{\domain} f_0 \,\varphi (0, \cdot) = 0,
\]
then the identity holds for every \(\varphi \in C^1_c ([0, +\infty) \times \Bar{D})\).
\end{proposition}
\begin{proof}
\resetconstant
We consider a map \(\theta \in C^1 ((0, + \infty))\) such that \(\theta = 0\) on \((0,\frac{1}{2})\) and \(\theta (t) = 1\) on \([1, +\infty)\) and we define for each \(n \in \integers\), the function \(\chi_n : D \to \reals\) for each \(x \in \domain\) by \(\chi_n (x) \defeq \theta (n \dist (x, \partial D))\). 
By the smoothness assumption on \(\domain\), \(\chi_n \in C^1_c (\domain)\).
Since \(\velocity \cdot \normal = 0\) in the sense of traces, we have for every \(T \in [0, +\infty)\),
\[
\int_0^T
  \int_{\domain} \abs{\velocity \cdot \nabla \chi_n}
  \le \C \int_0^T \int\limits_{\substack{x \in \domain\\ \dist (x, \partial \domain) \le \frac{1}{n}}} \abs{\nabla \velocity},
\]
and thus by Lebesgue's dominated convergence theorem, 
\begin{equation}
\label{eq_Saj2os5pou}
  \lim_{n \to \infty} \int_0^{T} \int_{\domain} \abs{\velocity \cdot \nabla \chi_n} =0.
\end{equation}

For each \(\varphi \in C^1_c ([0, +\infty) \times \Bar{\domain})\) and every \(n \in \integers\), we take \(\chi_n \varphi \in C^1_c ([0, +\infty) \times \domain)\) as test function, and we obtain by assumption
\[
    \int_0^{+\infty}
      \int_{\domain} 
        \chi_n
        \,
        f
        \,
        (\partial_t \varphi + \velocity \cdot \nabla \varphi) 
      + 
        \int_{\domain} 
          \chi_n 
          \, 
          f_0 
          \,
          \varphi (0, \cdot) 
  = 
    - \int_0^{+\infty}
      \int_{\domain} 
        \varphi\, f\, \velocity \cdot \nabla \chi_n;
\]
the conclusion follows by letting \(n \to \infty\), and using \eqref{eq_Saj2os5pou}.
\end{proof}

The flow \(\velocity\) can be integrated following DiPerna and P.-L. Lions \cite{DiPerna_Lions_1989} in order to provide solutions to the corresponding transport problem.

\begin{proposition}
\label{proposition_DiPerna_Lions}
Let \(b \in C^1(\Bar{\domain}, (0, +\infty))\) and assume that the velocity
field satisfies \(\velocity \in L^1_{\mathrm{loc}} ([0, +\infty), W^{1, 1} (\domain) \cap L^\infty (\domain))\).
If \(\velocity \cdot \normal = 0\) in the sense of traces and if \(\divergence (b\, \velocity) = 0\) in \(\domain\) almost everywhere, then there is a unique Borel-measurable function 
\(X : [0, +\infty) \times [0, +\infty) \times \domain \to \domain\), such that 
\begin{enumerate}[(i)]
  \item \label{it_Hah8arohr0} the map \((s, t) \in [0, +\infty)^2 \mapsto X (s, t, \cdot)\) is continuous for the convergence in measure,
  \item \label{it_xoh8iethaC} for every \(r, s, t \in [0, + \infty)\) and almost every \(x \in \domain\), one has \(X (s, t, x)= X (s, r, X (r, t, x))\),
  \item \label{it_acoo0eLie0} for every function \(f_0 \in L^\infty (\domain)\) and every \(s, t \in [0, +\infty)\), one has
  \[
      \int_{\domain} 
        f_0 (X (s, t, x)) 
        \, 
        b (x) 
        \dif x 
    =  
      \int_{\domain} 
        f_0 (x) 
        \,
        b (x) 
        \dif x ,
  \]
 \item \label{it_hooseeSh3Y} for almost every \(x \in \domain\), 
\[
  X (s, t, x) = x + \int_t^s \velocity (r, X (r, t, x)) \dif r.
\]
\end{enumerate}
Moreover, 
for every \(f_0 \in L^\infty (\domain)\), \(f_0 \circ X(0, \cdot)\) is the unique function \(f \in L^\infty ([0, +\infty) \times \domain)\) 
that satisfies for every \(\varphi \in C^1_c ([0, +\infty)\times \domain)\), 
\[
   \int_0^{+\infty}\int_{\domain} f(\partial_t \varphi + \velocity \cdot \nabla \varphi) 
    + \int_{\domain} f_0 \,\varphi (0, \cdot) = 0
\]
and \(f \in C([0, +\infty), L^1 (\domain))\). 
\end{proposition}
As a corollary of the above representation formula, the potential vorticity
$\vortex(t)/\depth$ at any time $t\geq 0$ is a rearrangement of the
initial potential vorticity $\vortex_0/\depth$, in the sense of the
weighted Lebesgue measure $\dif\mu(x)=\depth(x)\dif x$.

Note that,
a priori, the statements only make sense when the function \(f_0\) is Borel measurable; the proposition implies then that \(X (s, t, \cdot)\) preserves Lebesgue null sets and thus allows one to extend the statement to Lebesgue-measurable functions.

\begin{proof}[Proof of \cref{proposition_DiPerna_Lions}]
We first observe that \(\divergence \velocity = \velocity \cdot \nabla (\ln b)\) almost everywhere on \(\domain\) and thus \(\divergence \velocity \in L^\infty (\domain)\). The existence and the properties \eqref{it_Hah8arohr0}, \eqref{it_xoh8iethaC} and \eqref{it_hooseeSh3Y} of \(X\) follow from the DiPerna--Lions theory \cite{DiPerna_Lions_1989}*{Theorem III.2}, as does the characterization of solutions to the transport equations and the continuity of the latter \cite{DiPerna_Lions_1989}*{Corollary II.2}.
By \cref{lemma_Transport_TestFunctions}, the transport equation holds for each test function \(\varphi \in C^1_c ([0, +\infty) \times \Bar{\domain})\).

Given \(f_0 \in L^\infty (\domain)\) as an initial data, we observe that the function \(f  : [0, +\infty) \times \domain \to \reals\) defined for each \((t, x) \in [0, +\infty) \times \domain\) by \(f (t, x)\defeq f_0 (X(t, 0, x))\) satisfies the transport equation. By taking \(\varphi \in C^1_c ([0, +\infty) \times \Bar{\domain})\) defined for each \(t, x \in [0, +\infty) \times \Bar{\domain})\) by \(\varphi (t, x) \defeq b (x) \theta (t)\), with \(\theta \in C^1_c ((0, + \infty))\) as test function we have 
\begin{multline}
    \int_{0}^{+\infty} \theta' (t) \biggl(\int_{\domain} f_0 (X (t, 0, \cdot)) \, b\biggr) \dif t
    + \theta (0) \int_{\domain} f_0 \,b 
  \\
  =
    -\int_{0}^{+\infty} \biggl( \int_{\domain} f (X (t, 0, \cdot))\, \divergence (b\,\velocity (t)) \biggr)\dif t
  = 
    0,
\end{multline}
since \(\divergence (b\, \velocity (t)) = 0\) almost everywhere in \(\domain\) for almost every \(t \in [0, +\infty)\) and the conclusion follows.
\end{proof}

\section{Regularity of solutions with smooth initial data}
\label{appendix_regularity}
We prove that when the initial vorticity is smooth enough, then weak solutions of the vorticity formulation of the lake equations have some regularity.

\begin{proposition}
\label{proposition_classical}
Assume that \(k \in \integers\) and \(\alpha \in (0, 1)\), that \(\domain\) is of class \(C^{k + 1}\) and that \(b \in (C^2 \cap C^{k + 1, \alpha}) (\domain)\).
If \(\vortex_0 \in C^{k, \alpha} (\Bar{\domain}, \reals)\) and if  \((\vortex, \velocity) \in L^\infty ([0, +\infty) \times \domain, \reals) \times L^\infty( [0, +\infty), L^2 (\domain, \reals^2))\) is a weak solution to the vortex formulation of the lake equations, then for every \(T > 0\), \(\vortex \in C^{k, \alpha} ([0, T] \times \Bar{\domain})\) and 
\(\velocity \in L^\infty ([0, T], C^{k + 1, \alpha} (\bar{\domain}, \reals^2)) \cap C^{k, \alpha} ([0, T] \times \Bar{\domain}, \reals^2)\). 
\end{proposition}

When \(k = 0\), \cref{proposition_classical} is due to Huang \cite{Huang_2013}*{Theorem 4.1}.

Our proof follows the same strategy as proofs of the regularity of solutions of the planar Euler equations \cite{Marchioro_Pulvirenti_1994}*{\S 2.4} (see also \cite{Kato_1967}*{\S 3.1}).

The first tool that we use is the fact that the velocity field generated by a bounded vorticity field satisfies a bound known as quasi-Lipschitz bound \citelist{\cite{Kato_1967}*{Lemma 1.4}\cite{Marchioro_Pulvirenti_1994}*{Lemma 3.1}} or logarithmically Lipschitz \cite{Haroske_2000}.

\begin{lemma}
\label{lemma_log_Lipschitz}
There exists a constant \(C > 0\) such that for every \(\vortex \in L^\infty (\domain)\) and every \(x, y \in \domain\), one has 
\[
    \bigabs{
        \nabla\mathcal{K}_b [\vortex] (x) 
      -
        \nabla\mathcal{K}_b [\vortex] (y) 
    }    
  \le
    C\,
      \abs{y - x} \ln \frac{2\diam \domain}{\abs{y - x}}.
\]
\end{lemma}
\begin{proof}
\resetconstant
By \cref{velocityIntegralExpansion}, we have 
\begin{multline}
\label{eq_aiPahLoo4s}
    \nabla \mathcal{K}_b[\vortex] (x)
  =
    \int_{\domain} 
      \nabla G_D (x, z)  
      \,
      \vortex (z)
      \,
      \sqrt{b (x) \, b (z)} 
      \dif z
    \\
    + 
      \frac{1}{2} 
      \int_{\domain} 
        G_D (x, z) 
        \,
        \vortex (z) 
        \,
        \sqrt{\frac{b (z)}{b (x)}}
        \,
        \nabla b (x) 
        \dif z
    +
      \int_{\domain} 
        \nabla R_b (x, z) 
        \,
        \vortex (z) 
        \dif z.
\end{multline}
We first have the estimate
\begin{multline}
\label{eq_mooth7Nais}
    \biggabs{
        \int_{\domain} 
          \nabla G_D (y, z)  
          \,
          \vortex (z)
          \,
          \sqrt{b (y)\, b (z)} 
          \dif z -
        \int_{\domain} 
          \nabla G_D (x, z)  
          \,
          \vortex (z)
          \,
          \sqrt{b (x)\, b (z)} 
          \dif z
          }\\
  \le
    \C 
    \,
    \norm{\vortex}{L^\infty (\domain)}
    \,
    \abs{y - x} 
    \ln \frac{2\diam (\domain)}{\abs{y - x}}
\end{multline}
(see \cite{Marchioro_Pulvirenti_1994}*{Lemma 2.3.1 and Appendix 2.3}).
Next, we have 
\begin{multline}
\label{eq_aR3uodidei}
    \biggabs{
      \int_{\domain} 
          G_D (y, z)
          \,
          \vortex (z) 
          \,
          \sqrt{\frac{b (z)}{b (x)}}
          \,
          \nabla b (y) 
          \dif z
          -
        \int_{\domain}
          G_D (x, z)
          \,
          \vortex (z)
          \,
          \sqrt{\frac{b (z)}{b (x)}}
          \,
          \nabla b (x) 
          \dif z
    }
    \\
  \le 
      \frac{\abs{\nabla b (y) - \nabla b (x)}}{\sqrt{b (x)}}
      \,
        \int_{\domain} 
          G_D (x, z)\,
          \abs{\vortex (z)}
          \,
          \sqrt{b (z)}
          \dif z
        \\
    + 
      \C \frac{\abs{\nabla b (x)}}{\sqrt{b (x)}}
    \int_{\domain}  
      \bigabs{G_D (y, z) - G_D (x, z)} 
      \,
      \abs{\vortex (z)}
      \,
      \sqrt{b (z)}
      \dif z
      .
\end{multline}
We compute now by \cref{proposition_Gradient_bound},
\begin{equation}
\label{eq_aeK5eeGhia}
\begin{split}
 \int_{\domain} \abs{G_D (y, z) - G_D (x, z)}\dif z
 &\le \int_0^1 \int_{\domain} \abs{\nabla G_D ((1 - s) x + s y, z)} \,\abs{y - x}\dif z \dif t\\
  & \le 
    \int_0^1 
      \int_{\domain} 
        \frac{\C\, \abs{y - x}}{\abs{(1 - s) x + s y - z}} 
        \dif z 
      \dif s  
      \le
    \C 
    \,
    \abs{y - x}.
    \end{split}
\end{equation}
By \eqref{eq_aR3uodidei} and \eqref{eq_aeK5eeGhia}, we deduce, since the derivative of \(\nabla b\) is bounded, that the gap
\begin{equation}
\label{eq_raiThai9ch}
   \biggabs{
     \int_{\domain} 
        G_D (y, z) 
        \,
        \vortex (z) 
        \,
        \sqrt{\frac{b (z)}{b (y)}}
        \,
        \nabla b (y) 
        \dif z
   -
      \int_{\domain}
        G_D (x, z)
        \,
        \vortex (z)
        \,
        \sqrt{\frac{b (z)}{b (x)}}
        \,
        \nabla b (x) 
        \dif z
 }
\end{equation}
is bounded by
\[ 
  \C
  \,
  \norm{\vortex}{L^\infty (\domain)} 
  \,
  \abs{y - x}. 
\]
In order to control the variation of $\gradient R_b(\cdot,z)$,
we recall that by the proof of \cref{velocityIntegralExpansion}, $R_b=S_b+Q_b$ for some function $Q_b\in C^2(\domain\times\domain)$ defined in \cref{proposition_Green_stream_function}
and for some function $S_b$ constructed in the proof of \cref{proposition_b_Green_function} in such a way that for each
$z\in\domain$, the function \(S_b (\cdot, y) \in W^{1, 2}_0 (D)\) 
is the unique solution of the elliptic problem
\begin{equation*}
  \left\{
  \begin{aligned}
    -\divergence
      \bigl( 
        \depth^{-1}\gradient S_b (\cdot, z) 
      \bigr) 
  & =     
    -\greenlaplace(\cdot, z)
    \,
    \sqrt{\depth(z)}
    \,
    \Bigl(\Delta \frac{1}{\sqrt{b}}\Bigr)
    &
    &
    \text{in \(\domain\)},
    \\
      S_b (\cdot, z) 
    & =
      0
    &
    & \text{on \(\partial \domain\).}
  \end{aligned}
  \right.
\end{equation*}
Hence, in order to conclude the proof, it is sufficient to focus on the $S_b$-term.
We recall that for all $z\in\domain$
the function $S_b$ admits the integral representation
	\[ S_b(x,z)
		= -
    \sqrt{\depth(z)}\,\int_{\domain}G_b(y,x)\greenlaplace(y, z)
    \,
    \Bigl(\Delta \frac{1}{\sqrt{b}}\Bigr)(y)
    \dif y
    ;\]
and moreover, $S_b$ is continuous and symmetric on $\domain\times\domain$
(\cref{proposition_b_Green_function}). Therefore, we have for all $x,z\in\domain$:
	\[ S_b(x,z) = S_b(z,x)
		= -
    \sqrt{\depth(x)}\,\int_{\domain}G_b(y,z)\greenlaplace(y, x)
    \,
    \Bigl(\Delta \frac{1}{\sqrt{b}}\Bigr)(y)
    \dif y
    ,\]
or equivalently, using the symmetry of the Green's function $\greenlaplace$:
	\[ S_b(x,z) = -
    \,\int_{\domain}G_b(y,z)\Big(\sqrt{\depth}\ \greenlaplace(\cdot,y)\Big)(x)
    \,
    \Bigl(\Delta \frac{1}{\sqrt{b}}\Bigr)(y)
    \dif y
    .\]
In particular, a direct application of Fubini's theorem shows that, for almost-every
$x\in\domain$, we have
	\begin{equation} \gradient \int_\domain S_b(x,y)\vortex(y)\dif y
		= \int_\domain\vortex(y)\bigg(
	\int_\domain G_b(z,y)\Big(\Delta\frac{1}{\sqrt{\depth}}\Big)(z)
		\gradient\Big(\sqrt{\depth}\ \greenlaplace(\cdot,z)\Big)(x)\dif z
	\bigg)\dif y. \end{equation}
Since the $L^p$-norms of Green's functions are uniformly bounded on
as $y$ varies in $\domain$~\cite{Weinberger},
we may apply estimates \eqref{eq_mooth7Nais} and \eqref{eq_raiThai9ch}
to obtain
\begin{equation*}
  \biggabs{
    \int_{\domain} 
      \nabla S_b (x, z) 
      \,
      \vortex (z) 
      \dif z
    -
        \int_{\domain} 
      \nabla S_b (x, z)
      \,
      \vortex (z) 
      \dif z
  }
  \le 
    \C 
    \,
    \norm{\vortex}{L^\infty (\domain)}
    \,
    \abs{y - x} 
    \ln \frac{2\diam (\domain)}{\abs{y - x}}
    ,
\end{equation*}
and therefore
\begin{multline}
\label{eq_ieh2thi1Yi}
  \biggabs{
    \int_{\domain} 
      \nabla R_b (x, z) 
      \,
      \vortex (z) 
      \dif z
    -
        \int_{\domain} 
      \nabla R_b (x, z)
      \,
      \vortex (z) 
      \dif z
  }
  \le 
    \C 
    \,
    \norm{\vortex}{L^\infty (\domain)}
    \,
    \abs{y - x} 
    \ln \frac{2\diam (\domain)}{\abs{y - x}}
    .
\end{multline}
The conclusion follows from \eqref{eq_aiPahLoo4s}, \eqref{eq_mooth7Nais}, \eqref{eq_raiThai9ch} and \eqref{eq_ieh2thi1Yi}.
\end{proof}

The next tool is Grönwall type estimate for a logarithmic perturbation of linear growth.

\begin{lemma}
\label{lemma_log_Lipschitz_regularity}
Let \(A, B, C \in [0, +\infty)\), $A<C$ and let $f$ be a continuous function
from $[0,+\infty)$ to $(0,C)$, that is: \(f \in C([0,+\infty), (0, C))\).
If for every \(t \in [0, + \infty)\),
\[
  f (t) \le A + B \int_0^t f (s) \ln \frac{C}{f (s)} \dif s,
\]
then for every \(t \in [0, + \infty)\),
\[
  f (t)
  \le C \exp \Bigl(- \ln \tfrac{C}{A} e^{-B t}\Bigr).
\]
\end{lemma}
\begin{proof}
We observe that if the function \(u \in C^1 (\reals,(0, +C))\) satisfies for every \(t \in I\) the equation
\[
 u'(t) = B \, u (t) \ln \frac{C}{u (t)},
\]
then 
\[
 u(t) 
 = C \exp \bigl(- e^{-Bt}\ln \tfrac{C}{u (0)} \bigr)
\]
and the conclusion follows then by comparison.
\end{proof}

We finally rely on the next classical regularity property of Lagrangian flows.

\begin{lemma}
\label{lemma_flow_regularity}
If \(\velocity \in C^{k - 1} ([0, +\infty)\times \Bar{\domain}, \reals^2) \cap C ([0, +\infty), C^{k} (\Bar{\domain}, \reals^2))\) and if the function \(X \in C^1 ([0, +\infty), C(\Bar{\domain}, \Bar{\domain}))\) satisfies
\[
  \left\{
  \begin{aligned}
    \partial_t X (t, x) &= \velocity (t, X (t, x)) & &\text{if \(t \in [0, +\infty)\) and \(x \in \domain\)},\\
    X (0, x) & = x && \text{if \(x \in \domain\)},
  \end{aligned}
  \right.
\]
then \(X \in C^k ([0, +\infty) \times \domain)\).
If moreover \(\velocity \in C^{k - 1, \alpha} ([0, +\infty)\times \Bar{\domain}, \reals^2) \cap L^\infty ([0, +\infty), C^{k, \alpha} (\Bar{\domain}, \reals^2))\), then \(X \in C^{k, \alpha} ([0, +\infty) \times \domain)\).
\end{lemma}

\begin{proof}
The first part is classical (see for example \cite{Coddington_Levinson_1955}*{\S 1.7}). For the second part, we first have by the first part \(X \in C^{k} (\reals \times \Bar{\domain}, \Bar{\domain})\) and thus by the chain rule for H\"older continuous functions (see for example \cite{Csato_Dacorogna_Kneuss_2012}*{Theorem 16.31}) \(\partial_t X \in C^{k - 1, \alpha} (\reals \times \Bar{\domain}, \Bar{\domain})\).

Next we observe that for every \(T > 0\) and \(t \in [0, T]\), we have 
\begin{multline*}
  \abs{D_x^k X (t, y) - D_x^k X (t, x)}\\
  \le \int_0^t  \abs{D^k f (X (s, y)) [D_x^k X (s, y)]
  - D^k f (X (s, x)) [D_x^k X (s, x)]} \dif s
  + \C \abs{y - x}\\
   \le \C \int_0^t  \abs{D_x^k X (s, y)
  - D^k_x X (s, x)} \dif s
  + \C \abs{y - x}^\alpha,
\end{multline*}
and it follows then from the classical Grönwall inequality that 
\[
  \abs{D_x^k X (t, y) - D_x^k X (t, x)} \le \C \abs{y - x}^\alpha.
  \qedhere 
\]
\end{proof}

\begin{proof}[Proof of \cref{proposition_classical}]
By \cref{propositionRegularity}, \cref{proposition_DiPerna_Lions} is applicable to \(f_0 = \vortex_0/b\) and implies that for every \(t \in \reals\),
\[
    \vortex (t, x) 
  = 
    \frac{b (x)}{b (X (t, x))} 
    \, 
    \vortex_0 (X (t, x)).
\]
By \cref{proposition_DiPerna_Lions} and \cref{lemma_log_Lipschitz}, we have for every \(x, y \in \domain\),
\begin{multline*}
  \abs{X (t, y) - X (t, x)}\\
 \le 
    \abs{y - x} 
  + 
    \C 
    \,
    \Bigg(
    \bigl(\norm{\vortex_0}{L^\infty (\domain)} + \norm{\Gamma}{}\bigr)
    \,
    \int_0^t \abs{X (s, y) - X (s, x)} \ln \frac{2 \diam \domain}{\abs{X (s, y) - X (s, x)}}\dif s\Bigg).
\end{multline*}
It follows then by \cref{lemma_log_Lipschitz_regularity}, that 
\begin{equation*}
 \abs{X (t, y) - X (t, x)}
 \le
 \C
 \exp\Bigl(- \alpha e^{-\Cl{cst_Eemooph9qut} t} \ln \frac{2 \diam \domain}{\abs{y - x}})
 \Bigr)
 =\Cl{cst_Baara0gief} \Bigl(\frac{\abs{y - x}}{2 \diam D}\Bigr)^{ \alpha \exp({- \Cr{cst_Eemooph9qut} t})}.
\end{equation*}
This implies thus that 
\[
 \abs{\vortex (t, y) - \vortex (t, x)}
 \le 
 \C \Bigl(\frac{\abs{y - x}}{2 \diam D}\Bigr)^{\alpha \exp({-\Cr{cst_Eemooph9qut} t})}.
\]
By the representation formula for the velocity \eqref{eq_velocity_reconstruction} and classical regularity estimates \cite{Gilbarg_Trudinger_1998}*{Theorem~6.8}, it follows then that we have the inclusion \(\velocity \in C ([0, +\infty), C^1 (\domain))\).
By classical regularity theory of the Lagrangian flow, this implies that \(X \in C^1 ([0, + \infty) \times \Bar{\domain}, \Bar{\domain})\) and thus by composition \(\vortex \in C^{0, \alpha} ([0, T] \times \domain)\).
By regularity estimates \cite{Gilbarg_Trudinger_1998}*{Theorem~6.8} we have then \(\velocity \in L^\infty ([0, T], C^{1, \alpha} (\bar{\domain}))\) and \(u \in C^{0, \alpha} ([0, T] \times \Bar{\domain})\).

We assume now that \(\velocity \in L^\infty ([0, T], C^{k, \alpha} (\bar{\domain})) \cap C^{k - 1, \alpha} ([0, T] \times \Bar{\domain})\) and that \(\vortex_0 \in C^{k, \alpha} (\Bar{\domain})\). By regularity of the Lagrangian flow (\cref{lemma_flow_regularity}), we have \(X \in C^{k, \alpha} (\domain)\) and thus \(\vortex \in C^{k, \alpha} ([0, T] \times \domain)\). By classical regularity estimates, this implies that  \(\velocity \in L^\infty ([0, T], C^{k + 1, \alpha} (\bar{\domain})) \cap C^{k, \alpha} ([0, T] \times \Bar{\domain})\).
\end{proof}

\end{document}